\documentclass [12pt]{article}
\usepackage{amsmath,amssymb}
\usepackage{amssymb,verbatim}

\usepackage{enumerate}

\newcommand{\interior}[1]{%
  {\kern0pt#1}^{\mathrm{o}}%
}
\newcommand{\pdiv}{\mid\!\mid}
\usepackage{float}
\usepackage{enumitem}
\usepackage{multirow}
\usepackage{array}
\newcolumntype{L}[1]{>{\raggedright\let\newline\\\arraybackslash\hspace{0pt}}m{#1}}
\newcolumntype{C}[1]{>{\centering\let\newline\\\arraybackslash\hspace{0pt}}m{#1}}
\newcolumntype{R}[1]{>{\raggedleft\let\newline\\\arraybackslash\hspace{0pt}}m{#1}}
\usepackage{blkarray}
\usepackage{amsmath}
\usepackage{tikz-cd} 
\usepackage{subfig}
\usepackage{faktor}
\usepackage{xfrac}
\usepackage{faktor} 
\usepackage{amsmath} 
\usepackage{amssymb}
\usepackage{mathtools}

\usepackage[utf8]{inputenc}
\usepackage{graphicx}
\usepackage{graphicx,amssymb,amstext,amsmath}
\usepackage{tikz}
\usepackage{amsmath}
\usepackage{amssymb}
\usepackage{amsthm}
\usepackage{hebcal}
\usetikzlibrary{arrows, decorations.markings}
\usepackage{blindtext}
\usepackage{amsthm}
\usepackage[T1]{fontenc}
\usepackage[utf8]{inputenc}
\usepackage{blindtext}
\usepackage[T1]{fontenc}
\usepackage[utf8]{inputenc}
\newcommand\Item[1][]{%
  \ifx\relax#1\relax  \item \else \item[#1] \fi
  \abovedisplayskip=0pt\abovedisplayshortskip=0pt~\vspace*{-\baselineskip}}
\usepackage{tikz}
\usetikzlibrary{positioning, arrows.meta}

\usepackage{cite}
\newtheorem{thm}{Theorem}[section]
\newtheorem{lemma}[thm]{Lemma}

\newtheorem{example}[thm]{Example}
\newtheorem{remark}[thm]{Remark}

\newtheorem{thmx}{Theorem}
\newtheorem{question
}{Question}
\newtheorem*{acknowledgement}{Acknowledgement}
\def\0{{\bf 0}}

\def\N{{\bf N}}

\def\Q{{\bf Q}}
\def\R{{\bf R}}
\def\Z{{\bf Z}}

\usepackage{subfig}
\makeatletter
\def\keywords{\xdef\@thefnmark{}\@footnotetext}
 \title {Kepler Sets of Second-Order Linear Recurrence Sequences Over $\Q_p$}
\date{}
\author{Rishi Kumar \footnote{
Department of Mathematics, Ben-Gurion
University, Beer Sheva 84105, Israel.
E-mail: kumarr@post.bgu.ac.il}}
\begin{document}
\maketitle
\begin{abstract}
Let $(a_n)_{n=0}^\infty$ be a second-order linear recurrence sequence with constant coefficients over the field of $p$-adic numbers $\Q_p$. We study the set of limit points of the sequence of consecutive ratios $(a_{n+1}/a_{n})_{n=0}^\infty$ in $\Q_p$.
\end{abstract}

\section{Introduction}
Given a set $A$ of numbers (in any field), the {\it ratio set} of $A$ is
$$R(A)= \{a/a' \,|\, a,a' \in A\}.$$ 
Most studies on ratio sets deal with the real case (see, for example, \cite{Strauch} and the references therein). Here we consider the $p$-adic case.
The following theorem is due to Garcia and Luca:
\begin{thmx}\emph{\cite[Theorem 2]{Luca}} \label{Luca theorem on fibonacci}
Let $(F_n)_{n=0}^\infty$ be the Fibonacci sequence. Then the set $R(\{F_n\, | \, n\geq 1\})$ is dense in $\Q_p$ for all primes~$p$.
\end{thmx}

Sanna \cite{Sanna} extended Theorem \ref{Luca theorem on fibonacci} to the $k$-generalized Fibonacci sequence, defined by the recurrence
$$a_n=a_{n-1}+ \cdots +a_{n-k}, \qquad n\geq k,\, (k\geq 2,\, a_0=a_1=\cdots =a_{k-2}=0,a_{k-1}=1).$$

Garcia et al.\ \cite{GarciaLucaSanna} studied ratio sets of integer-valued linear recurrence sequences with constant coefficients of order 2,
\begin{equation}\label{intro, second order equation}
 a_{n}= ra_{n-1}+ sa_{n-2}, \qquad n\geq 2,  \end{equation}
for fixed $r,s\in \Z$ over $\Q_p$. For example, for the initial conditions $a_0=0, a_1=1$, they showed that, if $p\mid s$ and $p\nmid r$, then the set $R(\{a_n\,|\, n\geq 1\})$ is not dense in $\Q_p$, while if $p\nmid s$, then the set $R(\{a_n\,|\, n\geq 1\})$ is dense in $\Q_p$. For $a_0=2, a_1=r$, they showed that for any odd prime $p$ the set $R(\{a_n\,|\, n\geq 1\})$ is dense in $\Q_p$ if and only if there exists a positive integer $n$ such that $p\mid a_n$.

For more results concerning $R(A)$ for sets $A\subseteq \Z$ over $\Q_p$, we refer to \cite{Miska, GarciaLucaSanna, Donnay, Miska2, Paola} and the references therein.

Consider a recurrence sequence of order 2 as in \eqref{intro, second order equation}, but over any topological field $L$ (with arbitrary $r,s\in L$). A subset of $R(\{a_n\,| \, n\geq 0\})$ of special interest is the set of consecutive ratios $\{a_{n+1}/a_{n}\,|\, n\geq 0\}$. The limit $\lim_{n\to \infty}a_{n+1}/a_n$, if it exists, is the {\it Kepler limit} of $(a_n)_{n=0}^\infty$ (see \cite{Fiorenza}). In general, the {\it Kepler set} of $(a_n)_{n=0}^\infty$ is the closure of the set  $\{a_{n+1}/a_n\,:\, n\geq 0\}$ (see \cite{Berend1}).

The consecutive ratios $a_{n+1}/a_{n}$ of linear recurrence sequences with constant coefficients (mostly for order 2) over the fields $\R$ and ${\bf C}$ have been studied by numerous authors (see \cite{Fiorenza, Berend1, Poincare} and the references therein). Usually, given a linear recurrence sequence (of any order) with constant coefficients, the Kepler limit exists \cite[Theorem 2.3]{Fiorenza}. In \cite{Fiorenza, Berend1}, Kepler sets were studied for sequences as in \eqref{intro, second order equation}, with any real or complex coefficients $r$ and $s$, in the case where the Kepler limit does not exist. It turns out that, excluding the trivial case where $a_{n+1}/a_{n}$ is periodic (and so the Kepler set is finite), the Kepler set is either a line or a circle in the complex plane \cite[Theorem 2.2.1]{Berend1}.

In this paper, we study Kepler sets of sequences as in \eqref{intro, second order equation} over the field $\Q_p$ for primes $p\geq3$ (with arbitrary $r,s \in \Q_p$). Consecutive ratios $a_{n+1}/a_{n}$ over $\Q_p$ behave very differently than over ${\bf C}$. Thus, in Theorem \ref{Luca theorem on fibonacci}, for ``most'' primes $p$, it suffices to take two (specific) rows of the infinite square $\{F_{m}/F_{n}\, | \, m,n=1,2,3,\ldots\}$ to get density in $\Q_p$ (Theorem \ref{fibanacci theorem2} below). Moreover, for some $p$-s the Kepler set is already dense in $\Q_p$ (see Theorem \ref{result, proposition on Kepler set equal to Qp}). Whether or not there are infinitely many such primes is an open question. On the other hand, the Kepler set of the Fibonacci sequence $(F_n)_{n=0}^\infty$ over ${\bf C}$ is just a singleton: $\{(1+\sqrt{5})/2\}$.

In Section 2 we state our main results. We also characterize certain monothetic subgroups of the group of units of $\Q_p$ and of quadratic extensions thereof. In Section 3 we recall some preliminaries needed in the main proofs. Section 4 presents some examples. Section 5 is devoted to the proofs of the main results.

\begin{acknowledgement}\emph{
The author thanks D. Berend for suggesting this problem and for constant support throughout this work. The author thanks V. Anashin for discussions on $p$-adic Möbius transformations and P. Trojovsk\'{y} for information on the primes in sequence A000057 of the OEIS \cite{Sloane}. The author thanks the referee for many useful comments and suggestions regarding the results in the case of a ramified extension -- Theorems \ref{theorem for p-adic Kepler set2 of order 2 when roots are in Q_p}, \ref{result, proposition on Kepler set equal to Qp}, and \ref{fibanacci theorem2}.}
\end{acknowledgement}

\section{Main Results}\label{main result section}
Let $(a_n)_{n=0}^\infty$ be as in \eqref{intro, second order equation} with arbitrary $r,s\in \Q_p$, where $s\neq 0$.
The polynomial 
\begin{equation*}
    P(x)= x^2-rx-s
\end{equation*}
is the characteristic polynomial of $(a_n)_{n=0}^\infty$. Let $\lambda_1$ and $\lambda_2$ be the roots of $P$. We may rewrite $a_n$ as follows:
\begin{enumerate}
\item If $\lambda_1= \lambda_2= \lambda$, then
\begin{equation}\label{result, second order equation general form1}
    a_n= (c_1 + c_2 n )\lambda^n, \qquad n=0,1,2,\ldots, \, (c_1,c_2 \in \Q_p).
\end{equation}    
\item If $\lambda_1\neq \lambda_2$, then
\begin{equation}\label{result, second order equation general form2}
    a_n= c_1\lambda_1^n + c_2\lambda_2^n, \qquad n=0,1,2,\ldots, \, (c_1,c_2 \in \Q_p(\sqrt{d})),
\end{equation}
\end{enumerate}
(see \cite[p.4]{Graham}).

Denote the Kepler set by 
$$K((a_n))=\overline{\left\{a_{n+1}/a_{n}: n\geq 0\right\}},$$
and simply $K$ when the context is clear.

\begin{thm}\label{theorem, when single root}
Let $(a_n)_{n=0}^\infty$ be a second-order linear recurrence sequence over $\Q_p$, given by $a_n= (c_1 + c_2 n)\lambda^n$, with $c_1,c_2, \lambda \in \Q_p$, and $c_2,\lambda\neq 0$.
The Kepler set of $(a_n)_{n=0}^\infty$ is
\begin{equation}
    \begin{split}
     K=   \begin{cases}
            \lambda\left(1 +c_2/c_1+ (c_2/c_1)^2\cdot \Z_p\right),& \qquad |c_1|_p>|c_2|_p, \\
            \lambda \left(1 +  p\Z_p\right)^C,&\qquad |c_1|_p\leq|c_2|_p,
        \end{cases}
    \end{split}
\end{equation}
where $A^C$ denotes the complement of a set $A$.
\end{thm}
Let $(a_n)_{n=0}^\infty$ be as in \eqref{result, second order equation general form2}. We have
\begin{equation}\label{result sec, equation of ratios}
\frac{a_{n+1}}{a_{n}}= \frac{c_1\lambda_1^{n+1} +c_2 \lambda_2^{n+1}}{c_1\lambda_1^{n} + c_2\lambda_2^{n}}= \frac{c_1\lambda_1 + c_2\lambda_2(\lambda_2/\lambda_1)^n}{c_1 + c_2(\lambda_2/\lambda_1)^{n}}, \qquad n=0,1,2,\ldots.   
\end{equation}
If $|\lambda_1|_p\neq |\lambda_2|_p$, then by \eqref{result sec, equation of ratios} the limit $\lim_{n\to\infty}a_{n+1}/a_{n}$ exists and is either $\lambda_1$ or $\lambda_2$.
If $\lambda_1\neq \lambda_2$ and $\lambda_2/\lambda_1$ is a primitive root of unity of order $l$, then the sequence $(a_{n+1}/a_{n})_{n=0}^\infty$
is periodic with period $l$, and in particular, the Kepler set is of size~$l$. 

Our results deal with the non-trivial cases. Let us first recall a few facts about quadratic extensions of $\Q_p$.
Let $\Q_p(\sqrt{d})$ be a quadratic extension of $\Q_p$. We mention that, in fact, there are three different such extensions for $p\geq 3$, namely $\Q_p(\sqrt{N_p})$, $\Q_p(\sqrt{p})$, and $\Q_p(\sqrt{pN_p})$, where $N_p$ is the smallest positive integer that is not a square modulo $p$ (see \cite[p.72]{Mahler}). Thus, we will always assume that $d\in \Z_p^*\cup p\Z_p^*$.
The extension $\Q_p(\sqrt{d})$ is {\it ramified} if there exists an element $\pi \in \Q_p(\sqrt{d})$ such that $|\pi|_p= p^{-1/2}$; otherwise, it is {\it unramified}. In the ramified case, every non-zero element $x\in \Q_p(\sqrt{d})$ can be represented in the form
$$x= p^{\nu_p(x)}(x_0+ x_1\pi + x_2\pi^2+\cdots), 
\qquad (0\leq x_i\leq p-1,\, x_0\neq 0).$$
In this case it will be convenient to put $\nu_{\pi}(x)= 2\nu_{p}(x)$, so that 
$$x=\pi^{\nu_{\pi}(x)}(x_0+ x_1\pi + x_2\pi^2+\cdots).$$
In the unramified case, every non-zero element $x\in \Q_p(\sqrt{d})$ can be represented in the form
$$x= p^{\nu_p(x)}(z_0+ z_1\pi + z_2\pi^2+\cdots),\qquad (\pi=p),$$ 
where $z_i= x_i+ y_i\sqrt{d}$ and $0\leq x_i,y_i\leq p-1,\, x_0 + y_0\sqrt{d} \neq 0$. The fields $\Q_p$, for $p\geq 3$, have exactly one unramified quadratic extension $\Q_p(\sqrt{N_p})$. We have $\nu_p(\Q_p(\sqrt{d})^*)=\frac{1}{e}\cdot \Z$,
where $e=1$ in the unramifed case and $e=2$ in the ramified case; $e$ is the {\it ramification index} of $\Q_p(\sqrt{d})$ over $\Q_p$. Denote $f= 2/e$.

\begin{thm}\label{theorem for p-adic Kepler set of order 2 when roots are in Q_p}
Let $(a_n)_{n=0}^\infty$ be a second-order linear recurrence sequence over $\Q_p$, given by $a_n= c_1\lambda_1^n + c_2\lambda_2^n$, for a prime $p\geq 3$ and some non-zero $c_1,c_2, \lambda_1,\lambda_2$, where
$c_1,c_2, \lambda_1,\lambda_2$ belong to $\Q_p$ or to some unramified quadratic extension $\Q_p(\sqrt{d})$ thereof.
Suppose that $|\lambda_1|_p= |\lambda_2|_p>0$ and 
that $\lambda_2/\lambda_1$ is not a root of unity. Let $l$ be the order of $\lambda_2/\lambda_1$ modulo~$p$ and $k= \nu_p\left((\lambda_2/\lambda_1)^l-1\right)$.
\begin{enumerate}
\item If $-c_1/c_2 \notin \overline{\{(\lambda_2/\lambda_1)^n\,| \, n\in \N\}}$, then the Kepler set of $(a_n)_{n=0}^\infty$ is 
\begin{equation}\label{result sec, Kepler set in compact case}
    \begin{split}
        K&= \bigsqcup_{i=0}^{l-1}\left(\frac{a_{i+1}}{a_{i}} + p^{\nu_p(c_1/c_2)+\nu_p(\lambda_2-\lambda_1)-2\nu_p(c_1/c_2+(\lambda_2/\lambda_1)^i)+k}\Z_p\right),
    \end{split}
\end{equation}
where $\bigsqcup$ denotes disjoint union.
\item If $-c_1/c_2 \in \overline{\{(\lambda_2/\lambda_1)^n\, |\, n\in \N\}}$ and $l=1$, then the Kepler set of $(a_n)_{n=0}^\infty$ is
$$K= \left(\frac{\lambda_1+ \lambda_2}{2} + p^{1+\nu_p(\lambda_2)} \Z_p\right)^C.$$
\item If $-c_1/c_2 \in \overline{\{(\lambda_2/\lambda_1)^n\, |\, n\in \N\}}$, $l\geq 2$, and $s\in\{0,1,\ldots,l-1\}$ is defined by the condition
$$(\lambda_2/\lambda_1)^s\equiv  -c_1/c_2\,(\textup{mod} \ p^{k}),$$
then the Kepler set of $(a_n)_{n=0}^\infty$ is 
\begin{equation}\label{result, equation for third case of theorem}
   \begin{split}
    K= \bigsqcup_{i=0, i\neq s}^{l-1} \left(\frac{a_{i+1}}{a_{i}} + p^{\nu_p(\lambda_2-\lambda_1)+k}\Z_p\right) \bigsqcup \left(p^{\nu_p(\lambda_2)+1-k}\Z_p^C\right).
\end{split}
\end{equation}
\end{enumerate}
\end{thm}

\begin{remark}\emph{
When the roots $\lambda_1$ and $\lambda_2$ are in $\Q_p(\sqrt{d})-\Q_p$, the coefficients $c_1$ and $c_2$ are conjugate over $\Q_p$ and therefore $|c_1/c_2|_p=1$. In this case, the first part of Theorem~\ref{theorem for p-adic Kepler set of order 2 when roots are in Q_p} takes a somewhat simpler form as we may omit $\nu_p(c_1/c_2)$ in the right-hand side on \eqref{result sec, Kepler set in compact case}.}
\end{remark}

\begin{remark}\emph{The balls in the union on the right-hand side of \eqref{result sec, Kepler set in compact case} may be of different sizes. For example, consider the sequence $a_n=4+ 7^n$ over $\Q_5$. By Theorem \ref{theorem for p-adic Kepler set of order 2 when roots are in Q_p}.$(1)$
$$K= \bigsqcup_{i=0}^3 B_i,$$
where 
$$B_0=1/5+ \Z_5,\qquad B_1= 23 + 5^2\Z_5,\qquad B_2= 24 + 5^2\Z_5,\qquad B_3= 15 + 5^2\Z_5.$$
The radius of $B_0$ is 1, while that of the others is $1/5^2$.}
\end{remark}

\begin{thm}\label{theorem for p-adic Kepler set2 of order 2 when roots are in Q_p}
Let $(a_n)_{n=0}^\infty$ be a second-order linear recurrence sequence over $\Q_p$, given by $a_n= c_1\lambda_1^n + c_2\lambda_2^n$, for a prime $p\geq 3$ and some non-zero $c_1,c_2, \lambda_1,\lambda_2,$ where $\lambda_1,\lambda_2\in \Q_p(\sqrt{d})-\Q_p$ and the extension $\Q_p(\sqrt{d})$ is ramified. Suppose that $|\lambda_1|_p= |\lambda_2|_p>0$ and 
that $\lambda_2/\lambda_1$ is not a root of unity. Let $l$ be the order of $\lambda_2/\lambda_1$ modulo $\pi$, and let $k$ be such that $(\lambda_2/\lambda_1)^l= 1 + \pi^k \mu$ with $\pi\nmid \mu$.

\begin{enumerate}
\item If $-c_1/c_2 \notin \overline{\{(\lambda_2/\lambda_1)^n\,| \, n\in \N\}}$, then the Kepler set of $(a_n)_{n=0}^\infty$ is 
\begin{equation}\label{result sec, Kepler set in compact case2}
    \begin{split}
        K&= \bigsqcup_{i=0}^{l-1}\left(\frac{a_{i+1}}{a_{i}} + p^{\lceil\nu_p(\lambda_2-\lambda_1)-2\nu_p(c_1/c_2+(\lambda_2/\lambda_1)^i)+k/2\rceil}\Z_p\right).
    \end{split}
\end{equation}
\item If $-c_1/c_2 \in \overline{\{(\lambda_2/\lambda_1)^n\, |\, n\in \N\}}$ and $l=1$, then the Kepler set of $(a_n)_{n=0}^\infty$ is
$$K= \left(\frac{\lambda_1+ \lambda_2}{2} + p^{\lceil\nu_p(\lambda_2)\rceil +1} \Z_p\right)^C.$$
\item If $-c_1/c_2 \in \overline{\{(\lambda_2/\lambda_1)^n\, |\, n\in \N\}}$, $l\geq 2$, and $s\in\{0,1,\ldots,l-1\}$ is defined by the condition
$$(\lambda_2/\lambda_1)^s\equiv  -c_1/c_2\,(\textup{mod} \ \pi^{k}),$$
then the Kepler set of $(a_n)_{n=0}^\infty$ is 
\begin{equation}\label{result, equation2 for third case of theorem}
   \begin{split}
    K= \bigsqcup_{i=0, i\neq s}^{l-1} \left(\frac{a_{i+1}}{a_{i}} + p^{\lceil\nu_p(\lambda_2-\lambda_1)+(k/2)\rceil}\Z_p\right) \bigsqcup \left( p^{\lceil \nu_p(\lambda_2)+1-(k/2)\rceil}\Z_p^C\right).
\end{split}
\end{equation}
\end{enumerate}
\end{thm}

\begin{remark}\label{result, lemma on k odd} \emph{
In the setup of Theorem \ref{theorem for p-adic Kepler set2 of order 2 when roots are in Q_p}, let $\lambda_1= a+ b\sqrt{d}$ and $\lambda_2=a-b \sqrt{d}$, with $a,b\in \Q_p$. If $\nu_p(a)\leq\nu_p(b)$, then
$$\frac{\lambda_2}{\lambda_1}-1= \frac{a-b\sqrt{d}}{a+b\sqrt{d}}-1= \frac{-2b\sqrt{d}/a}{1+ b\sqrt{d}/a}= \pi^{k_1}\mu_1,\qquad \left(\pi\nmid \mu_1,\, \mu_1\in \Q_p(\sqrt{d})\right),$$
where $k_1$ is a positive odd integer.
If $\nu_p(a)>\nu_p(b)$, then
$$\left(\frac{\lambda_2}{\lambda_1}\right)^2-1= \left(\frac{a-b\sqrt{d}}{a+b\sqrt{d}}\right)^2-1= \frac{-4a/(b\sqrt{d})}{(1+ a/(b\sqrt{d}))^2}= \pi^{k_2}\mu_2,\qquad \left(\pi\nmid \mu_2,\, \mu_2\in \Q_p(\sqrt{d})\right),$$
where $k_2$ is a positive odd integer.
Hence, $l$ is at most 2, and $k$ is a positive odd integer. In \eqref{result, equation2 for third case of theorem}, the first union actually consists of one set.
}
\end{remark}

Let $a, b$ be non-zero integers. A \emph{Lucas sequence of the first kind} is a sequence $(L_n)_{n=0}^\infty = L(a,b)$, defined by the initial conditions $L_0=0$, $L_1=1$, and the recurrence
$$L_n = a L_{n-1}+ b L_{n-2}, \qquad n\geq 2$$
for some integers $a, b$ (see \cite[p.2]{Ribenboim2}). Note that the Fibonacci sequence is a special case where $a=b=1$.
The roots of the characteristic polynomial of $L(a,b)$ are
$$\lambda_1= \frac{a+ \sqrt{D}}{2}, \qquad \lambda_2 =\frac{a- \sqrt{D}}{2},$$
where $D= a^2 +4b$ is the discriminant. By Binet's formula we have
$$L_n= \frac{\lambda_1^n -\lambda_2^n}{\lambda_1 - \lambda_2}, \qquad n=0,1,2,\ldots.$$
If the ratio $\lambda_2/\lambda_1$ is a root of unity, the sequence $L(a,b)$ is \emph{degenerate}. It is known that, if 
$$(a,b) = (\pm 2,1), (\pm 2,-1), (0,\pm1), (\pm 1, 0),$$
then the sequence $L(a,b)$ is degenerate (see \cite[p.5-6]{Ribenboim2}, \cite{Sanna4}).
For any given prime $p$, the \emph{rank of appearance} \cite[p.10]{Ribenboim2} of $p$ in the Lucas sequence $(L_n)_{n=0}^\infty$ is the smallest positive integer $n$ such that $p\mid L_n$ ($\infty$ if $p\nmid L_n$ for every $n>0$). Denote the rank of appearance by $\alpha_{L}(p)$.

For every prime $p$, the rank of appearance of $p$ in the Fibonacci sequence $(F_n)_{n=0}^\infty$ satisfies $\alpha_{F}(p)\le p+1$ (see \cite{Paul}). A prime $p$ is a \emph{Wall-Sun-Sun prime} if $p^2 | F_{\alpha_{F}(p)}$ (see \cite[Section 5]{Pomerance}). The definition may be vacuous as there are no Wall-Sun-Sun primes at least up to $2^{64}$ (see \cite{Lenny}). However, there is a heuristic argument suggesting that there are, in fact, infinitely many such primes \cite[Section 4]{Richard}.

\begin{thm}\label{result, proposition on Kepler set equal to Qp}
Let $L(a,b)$ be a Lucas sequence of the first kind and $p\geq 3$ be a prime. Then:
\begin{enumerate}
\item If $\lambda_1,\lambda_2 \in \Q_p$, then $K((L_n))\neq \Q_p$.
\item If $\Q_p(\sqrt{D})$ is an unramified extension of $\Q_p$, $\alpha_{L}(p)= p+1$, and $\nu_p\left((\lambda_2/\lambda_1)^{\alpha_{L}(p)}-1\right)=~1$, then the Kepler set of $L(a,b)$ is the whole $p$-adic field: 
$$K((L_n))=\Q_p.$$
In particular, if $p$ is not a Wall-Sun-Sun prime and $\alpha_{F}(p) =p+1$, then
$$K((F_n))=\Q_p.$$
\item If $\Q_p(\sqrt{D})$ is a ramified extension of $\Q_p$ with $\nu_p(\lambda_1)=1/2$, and $(\lambda_2/\lambda_1)^l= 1 +\pi\mu$, $\pi\nmid \mu$, where $l\geq 2$ is the order of $\lambda_2/\lambda_1$ modulo $\pi$, then 
$$K((L_n))=\Q_p.$$
\end{enumerate}
\end{thm}

\begin{thm}\label{fibanacci theorem2}
Let $L(a,1)$ be a Lucas sequence of the first kind, let $p\geq 3$ be a prime, and $\alpha_L(p)>2$. Let $\lambda_1$ and $\lambda_2$ be the roots of the characteristic polynomial of $L(a,1)$. If~$p^2\mid D$ or $(\lambda_2/\lambda_1)^{l} -1$ is not divisible by $\pi^2$, where $l$ is the order of $\lambda_2/\lambda_1$ modulo $\pi$ ($\pi=p$ if $\Q_p(\sqrt{D})$ is an unramifed extension of $\Q_p$), then
$$\left(\overline{\left\{\frac{L_{n+\alpha_{L}(p)}}{L_n}\right\}_{n=0}^\infty}\right)\bigcup \left(\overline{\left\{\frac{L_{n+\alpha_{L}(p)-2}}{L_n}\right\}_{n=0}^\infty}\right)= \Q_p.$$
In particular, if $p$ is not a Wall-Sun-Sun prime, then
$$\left(\overline{\left\{\frac{F_{n+\alpha_{F}(p)}}{F_n}\right\}_{n=0}^\infty}\right)\bigcup \left(\overline{\left\{\frac{F_{n+\alpha_{F}(p)-2}}{F_n}\right\}_{n=0}^\infty}\right)= \Q_p.$$
\end{thm}

\begin{remark}\label{Remark K1 and K2}\emph{It follows from the proof of Theorem \ref{fibanacci theorem2} that if $p^2\mid D$, then
$$\left(\overline{\left\{\frac{L_{n+1}}{L_n}\right\}_{n=0}^\infty}\right)\bigcup \left(\overline{\left\{\frac{L_{n+2}}{L_n}\right\}_{n=0}^\infty}\right)= \Q_p.$$
}
\end{remark}

\begin{remark}\emph{
One should compare the preceding two theorems with the results of Garcia et al. \cite{GarciaLucaSanna}, mentioned in the introduction. They find, under mild conditions, that the full ratio set is dense in $\Q_p$. We show that, under somewhat stronger conditions, already the set of ratios of consecutive (Theorem \ref{result, proposition on Kepler set equal to Qp}) or some sets of ``almost consecutive'' ratios (Theorem \ref{fibanacci theorem2} and Remark \ref{Remark K1 and K2}) are dense.}
\end{remark}

How large is the set of primes for which the Kepler set of the Fibonacci sequence is $\Q_p$, that is, primes for which
$\alpha_{F}(p)=p+1$? (We ignore the condition that $p$ is not a Wall-Sun-Sun prime.) On the one hand, even the question of whether it is infinite is open \cite{Paul} (unless we assume the generalized Riemann Hypothesis; see below). On the other hand, numerical data strongly suggest that it is a large set of primes. Indeed, the sequence is one of the very first in the OEIS \cite[A000057]{Sloane}. According to the comments there, the number of such primes up to $10^k$ has been calculated by Greathouse for $1\leq k\leq 10$. In Table \ref{table Fibonacci zero}, we present the number of primes up to $10^k$, the number of those that meet the condition $\alpha_{F}(p)=p+1$, and the ratio between the two quantities for $1\leq k\leq 10$. The relative density of primes that satisfy the property appears to converge to $0.1968\ldots$.
\begin{remark}\emph{It follows from a very recent result of Sanna \cite{Sanna3} that, under the Generalized Riemann Hypothesis, the set of primes with $\alpha_{F}(p)=p+1$ has a relative density $10/19 A= 0.1968\ldots$, where $A=\prod_{p\in \mathcal{P}}\left(1- \frac{1}{p(p-1)}\right)$ is the Artin constant (see \cite[p.303]{Ribenboim}), and $\mathcal{P}$ denotes the set of all primes. This agrees with our numerical experiments.
}
\end{remark}
\begin{table}[!htbp] 
    \centering
\begin{tabular}{|   C{1cm} | C{3cm} | C{3cm} | C{2cm} |}
 \hline
$k$ & Number of primes in $[1,10^k]$&Number of primes satisfying $\alpha_{F}(p)=p+1$& Proportion\\
 \hline
 ~1 & ~~~~~~~~~~~~~~4    &~~~~~~~~~~~3&0.75~~~~~~~~\\
 \hline
 ~2  &~~~~~~~~~~~~~25    &~~~~~~~~~~~7&0.28~~~~~~~~\\
 \hline
~3  &~~~~~~~~~~~168    &~~~~~~~~~~38&0.2261\ldots\\
 \hline
~4  & ~~~~~~~~~~1229  &~~~~~~~~249&0.2026\ldots\\\hline
~5  &~~~~~~~~~~9592&~~~~~~1894&0.1975\ldots\\\hline
~6  &~~~~~~~~78498&~~~~~15456&0.1968\ldots\\\hline
~7  &~~~~~~664579&~~~130824&0.1968\ldots\\\hline
~8  & ~~~~~5761455&11344404&0.1969\ldots\\\hline
~9  &~~~50847534&10007875&0.1968\ldots\\\hline
10  & 4550525511&89562047&0.1968\ldots\\\hline
\end{tabular}
\caption{Proportion of primes satisfying $\alpha_{F}(p)=p+1$.}
\label{table Fibonacci zero}
\end{table}
The frequency of primes with $\alpha_{F}(p)=p+1$ seems pretty fixed. There are 179055 such primes among the first $10^6$ primes. Counting them in each of the 10 sub-intervals $(j\cdot10^5,(j+1)\cdot10^5]$, $0\leq j\leq 9$, of primes, we see that in each sub-interval, there are about one-tenth of them. In fact, the minimum is 19534, attained at $(8\cdot10^5,9\cdot10^5]$ and the maximum is 19781, attained at $(4\cdot10^5,5\cdot 10^{5}]$. This raises the question of whether, in fact, the relative density of the set of primes satisfying $\alpha_{F}(p)=p+1$ becomes arbitrarily close to $10/19 A$ for every sufficiently large finite set of consecutive primes. (This is analogous to the notion of \emph{Banach density}; see, for example,~\cite[p.72]{FURSTEN}.)

It will be convenient to use the notation $U$ instead of $\Z_p^*$ and $U(\sqrt{d})$ instead of $\Z_p
[\sqrt{d}]^*$. Define the norm $N$ on $U(\sqrt{d})$ by 
$$N(x+ y\sqrt{d})= x^2-y^2d,\qquad x, y\in \Q_p.$$
(Actually, $N$ is defined on the whole of $\Q_p(\sqrt{d})$, but we prefer to consider only its restriction to $U(\sqrt{d})$.)
Denote:
$$U^{0}(\sqrt{d})= \{z\in U(\sqrt{d})\,|\, N(z)=\pm 1\}.$$

Our results depend on being able to calculate the closed multiplicative subgroups generated by a single element in $U$ and $U(\sqrt{d})$ (in the latter case $-$ only by elements in $U^0(\sqrt{d})$). They are very related to some of the results of Coelho and Parry \cite{Parry}, who found the number of ergodic components of maps of the form $x\mapsto \lambda x$, $x\in U$ (or $x\in U^0(\sqrt{d})$). Theorems \ref{prop, orbit closure of 1 in Qp} and \ref{prop, orbit of 1 in ramified extension, under T(z)=lambda z} can probably be inferred from their proofs.  
(For more details on ergodic components of maps of the form $x\mapsto \lambda x$, we refer to \cite{Fan1, Fan2} and the references therein.)

\begin{thm} \label{prop, orbit closure of 1 in Qp} Let $p\geq 3$ be a prime, $\lambda \in U$ be of order $l$ modulo $p$, and $k= \nu_p(\lambda^l-1)$. Then the closed subgroup of $U$ generated by $\lambda$ is 
$$\overline{\left\{\lambda^n\,|\, n\in\N\right\}}= \bigsqcup_{i=0}^{l-1} \left(\lambda^i + p^k \Z_p\right).$$
\end{thm}
If $\lambda$ is a root of unity, then $k=\infty$, and the theorem is still correct.

\begin{thm}\label{prop, orbit of 1 in ramified extension, under T(z)=lambda z}
Let $\Q_p(\sqrt{d})$ be a quadratic unramified extension of $\Q_p$, where $p\geq 3$ is a prime. Suppose that $\lambda \in U^{0}(\sqrt{d})$ has order $l$ modulo $p$, and $k= \nu_p(\lambda^l-1)$. Then the closed subgroup of $U^0(\sqrt{d})$ generated by $\lambda$ is
$$\overline{\left\{\lambda^n\,|\, n\in \N\right\}}= \bigsqcup_{i=0}^{l-1} \left(\lambda^i + p^k \Z_p[\sqrt{d}]\right)\bigcap U^{0}(\sqrt{d}).$$
In particular, if $l= 2(p+1)$ and $k=1$, then
$$\overline{\left\{\lambda^n\,|\, n\in \N\right\}}= U^{0}(\sqrt{d}).$$
\end{thm}

\begin{thm}\label{prop, orbit of 1 in extension2, under T(z)=lambda z}
Let $\Q_p(\sqrt{d})$ be a quadratic ramified extension of $\Q_p$, where $p\geq 3$ is a prime. Suppose that $\lambda \in U^{0}(\sqrt{d})$ has order $l$ modulo $\pi$, and $k= \nu_{\pi}(\lambda^l-1)$. Then the closed subgroup of $U^0(\sqrt{d})$ generated by $\lambda$ is
\begin{equation}\label{result, equation of monothetic group in ramified case}
\overline{\left\{\lambda^n\,|\, n\in \N\right\}}= \bigsqcup_{i=0}^{l-1} \left(\lambda^i + \pi^k \Z_p[\sqrt{d}]\right)\bigcap U^{0}(\sqrt{d}).    
\end{equation}

In particular, if  $l=4$ and $k=1$, then 
$$\overline{\left\{\lambda^n\,|\, n\in \N\right\}}= U^{0}(\sqrt{d}).$$
\end{thm}

We require some notations and terminology from Section \ref{pre section} to provide examples corresponding to the theorems mentioned above. Thus, we present the examples only later, in Section~\ref{examples section}. However, the reader may find it useful to browse through them already at this point.

\section{Preliminaries}\label{pre section}
\subsection{The Exponential and Logarithm Functions Over $p$-adic Fields}
Let $K$ be either $\Q_p$ or a quadratic extension $\Q_p(\sqrt{d})$ thereof. For $a\in K$ and $r>0$, denote the open disk of radius $r$ centered at $a$ by
$$D(a,r)=\{x\in K\,|\, |x-a|_p<r\},$$
the closed disk of radius $r$ centered at $a$ by
$$\overline{D}(a,r)=\{x\in K\,|\, |x-a|_p\leq r\},$$
and the sphere of radius $r$ centered at $a$ by
$$S(a,r)=\{x\in K\,|\, |x-a|_p=r\}.$$

The $p$-adic logarithm function is defined by 
$$\log_p(x)= \sum_{n=1}^\infty (-1)^{n+1}\frac{(x-1)^n}{n}, \qquad x\in 1+ \pi O_K,$$
where $O_K$ is the ring of integers of $K$ ($\Z_p$ for $K= \Q_p$ and $\Z_p[\sqrt{d}]$ for $K=\Q_p(\sqrt{d})$).
For any $x,y \in 1+ \pi O_K$, we have 
$$\log_p(x y)= \log_p(x)+\log_p(y).$$
The $p$-adic exponential function is defined by 
$$\exp_p(x)= \sum_{n=0}^\infty \frac{x^n}{n!}, \qquad x\in D(0,p^{-1/(p-1)}).$$
For any $x,y \in D(0,p^{-1/(p-1)})$, we have $x+y \in D(0,p^{-1/(p-1)})$ and 
$$\exp_p(x+y)= \exp_p(x)\exp_p(y).$$
It can be seen that for $x\in D(0,p^{-1/(p-1)})$, we have
$$\log_p(\exp_p(x))= x, \qquad  \exp_p(\log_p(1+x))= 1+x.$$
Thus, $\exp_p$ is an isometric isomorphism from $D(0,p^{-1/(p-1)})$ onto $1+ D(0,p^{-1/(p-1)})$ and $\log_p$ is its inverse. For more details on the $p$-adic exponential and logarithm functions, we refer to \cite[Section 5.4]{Robert}.

\subsection{Some Monothetic Subgroups of the Group of Units}
Denote the subgroups $1+p^n\Z_p$ of $U$ by $U_n$, $n\geq 1$. Clearly
$U\supseteq U_1 \supseteq U_2\supseteq\cdots$ and $\bigcap_{n=1}^\infty U_n= \{1\}$.
The group $U(\sqrt{d})$ contains a subgroup $V(\sqrt{d})$ of order $(p^f-1)$, consisting of the roots of unity of this order. The group $U(\sqrt{d})$ factorizes in the form: $U(\sqrt{d})= V(\sqrt{d}) \times (1+ \pi \Z_p[\sqrt{d}])$ (see \cite[p.106]{Robert}). The subgroup $1+ \pi \Z_p[\sqrt{d}]$ contains the subgroups:
$$U_n(\sqrt{d})= 1+ \pi^n\Z_p[\sqrt{d}], \qquad n=1,2,3,\ldots.$$
Clearly, $$U(\sqrt{d})\supseteq U_1(\sqrt{d}) \supseteq U_2(\sqrt{d})\supseteq\cdots, \qquad  \bigcap_{k=1}^\infty U_k(\sqrt{d})= \{1\}.$$
Denote
$$U^{0}_n(\sqrt{d})= \{z\in U_n(\sqrt{d})\,|\, N(z)=1\}, \qquad n=1,2,3\ldots.$$
The following three lemmas present some relations between the subgroups $U_n^0(\sqrt{d})$, when the extension is unramified (Lemma \ref{monothetic group lemma unramified case}) and when it is ramified (Lemma \ref{monothetic group lemma ramified case} and Lemma \ref{monothetic group lemma ramified case second}). (The first part of Lemma \ref{monothetic group lemma unramified case} is explained in \cite[p.2]{Parry}. We are unaware of the second part, as well as Lemma \ref{monothetic group lemma ramified case} and Lemma \ref{monothetic group lemma ramified case second}, being stated in the literature.)
\begin{lemma}\label{monothetic group lemma unramified case}
Let $\Q_p(\sqrt{d})$ be an unramified quadratic extension of $\Q_p$, where $p\geq 3$ is a prime. Then:
\begin{enumerate}
\item $$|U^{0}(\sqrt{d})/U^{0}_1(\sqrt{d})|= 2(p+1);$$
\item $$|U^{0}_k(\sqrt{d})/U^{0}_n(\sqrt{d})|= p^{n-k}, \qquad n\geq k\geq 1.$$
\end{enumerate}
\end{lemma}
\begin{proof}
\begin{enumerate}
\item The canonical map $\Z_p[\sqrt{d}] \to \Z_p[\sqrt{d}]/p\Z_p[\sqrt{d}]\cong(\Z/p\Z)[\sqrt{d}]$ induces a surjective homomorphism 
$$ P_{0}: U(\sqrt{d}) \to (\Z/p\Z)[\sqrt{d}]^*.$$
We have $\ker(P_0)= U_{1}(\sqrt{d})$. Hence $$|U(\sqrt{d})/U_1(\sqrt{d})|= p^2-1.$$
Consider the commutative diagram in Figure \ref{commutative diagrams1}.
\begin{figure}
\centering
\begin{tikzcd}
&1 \arrow[d]     & 1 \arrow[d]    & 1 \arrow[d]&\\
1 \arrow[r]& U^{0}_1(\sqrt{d}) \arrow[d, "i"] \arrow[r, "i"]     & U_1(\sqrt{d}) \arrow[r,"N"] \arrow[d, "i"]    & U_1 \arrow[d, "i"]\arrow[ r]&1 \\
1 \arrow[r]& \ker(N) \arrow[d, "P"] \arrow[r, "i"]     & U(\sqrt{d}) \arrow[r,"N"] \arrow[d, "P"]    & U \arrow[d, "P"]& \\
1 \arrow[r]& \ker(N)/U^{0}_{1}(\sqrt{d}) \arrow[d] \arrow[r,"\overline{i}"]     & U(\sqrt{d})/U_{1}(\sqrt{d}) \arrow[r,"\overline{N}"] \arrow[d]    & U/U_{1} \arrow[d]& \\
&1                          & 1                       & 1 &\\
\end{tikzcd}
\caption{Relations between $U_1, U^{0}_1(\sqrt{d})$, and $U_1(\sqrt{d})$, for unramified extensions.}\label{commutative diagrams1}
\end{figure}
(Here and later, $i$ and $P$ denote embedding and projection, respectively; $\overline{i}$ and $\overline{N}$ are the homomorphisms induced by $i$ and $N$, respectively.) By the Snake Lemma \cite[Lemma 1.3.2]{Weibel} we have $\textup{Im}(\overline{i})= \ker(\overline{N})$. The cardinality of the group $U/U_1$ is $p-1$ (see \cite[p.223]{Narkiewicz}). It follows from \cite[Lemma 11]{Parry} that $\overline{N}$ is surjective. Therefore
$$|\ker(N)/U^{0}_1(\sqrt{d})|=p+1.$$ 
Hence,
$$|U^{0}(\sqrt{d})/U^{0}_1(\sqrt{d})|=2(p+1).$$ 
\item The norm homomorphism $N : U_n(\sqrt{d})\to U_n$, $n\in \N$, is surjective. In fact, for any $z\in U_n$ we have $w= \exp_p(1/2\cdot \log_p(z)) \in U_n$ and $N(w)=z$. For any $n\geq k$, consider the commutative diagram in Figure \ref{commutative diagrams2}.
\begin{figure}
\centering
\begin{tikzcd}
&1 \arrow[d]     & 1 \arrow[d]    & 1 \arrow[d]&\\
1 \arrow[r]& U^{0}_n(\sqrt{d}) \arrow[d, "i"] \arrow[r, "i"]     & U_n(\sqrt{d}) \arrow[r,"N"] \arrow[d, "i"]    & U_n \arrow[d, "i"]\arrow[ r]&1 \\
1 \arrow[r]& U^{0}_k(\sqrt{d}) \arrow[d, "P"] \arrow[r, "i"]     & U_k(\sqrt{d}) \arrow[r,"N"] \arrow[d, "P"]    & U_k \arrow[d, "P"]\arrow[ r]&1 \\
1 \arrow[r]& U^{0}_k(\sqrt{d})/U^{0}_{n}(\sqrt{d}) \arrow[d] \arrow[r,"\overline{i}"]     & U_k(\sqrt{d})/U_{n}(\sqrt{d}) \arrow[r,"\overline{N}"] \arrow[d]    & U_k/U_{n} \arrow[d]\arrow[ r]&1 \\
&1                          & 1                      & 1 &\\
\end{tikzcd}
\caption{Relations between $U^{0}_k(\sqrt{d})$ and $U^{0}_n(\sqrt{d})$, with $n\geq k$, for unramified extensions.}\label{commutative diagrams2}
\end{figure}
Invoking the Snake Lemma again, we see that $\textup{Im}(\overline{i})= \ker(\overline{N})$ and that $\overline{N}$ is surjective. It follows from \cite[Proposition 5.9]{Narkiewicz} that 
$$|U_n/U_{n+1}|= p, \qquad n=1,2,3,\ldots,$$
and
$$|U_n(\sqrt{d})/U_{n+1}(\sqrt{d})|=p^{2},  \qquad n=1,2,3,\ldots.$$
Therefore,
$$|U_k/U_n|= p^{n-k}, \qquad n\geq k\geq 1,$$
and 
$$|U_k(\sqrt{d})/U_{n}(\sqrt{d})|=p^{2(n-k)}, \qquad n\geq k\geq 1.$$
Hence,
$$|U^{0}_k(\sqrt{d})/U^{0}_n(\sqrt{d})|= p^{n-k}, \qquad n\geq k\geq 1.$$
\end{enumerate}
\end{proof}

\begin{lemma}\label{monothetic group lemma ramified case}
Let $\Q_p(\sqrt{d})$ be a ramified quadratic extension of $\Q_p$, where $p\geq 3$ is a prime. Then:
\begin{enumerate}
    \item $$|U^{0}(\sqrt{d})/U^{0}_2(\sqrt{d})|= 4p;$$
    \item $$|U^{0}_{2k}(\sqrt{d})/U^{0}_{2n}(\sqrt{d})|= p^{n-k}, \qquad n\geq k\geq 1.$$
\end{enumerate}
\end{lemma}
\begin{proof}
\begin{enumerate}
\item The canonical map $\Z_p[\sqrt{d}] \to \Z_p[\sqrt{d}]/\pi\Z_p[\sqrt{d}]\cong\Z/p\Z$ induces the surjective homomorphism 
$$ P_{0}: U(\sqrt{d}) \to (\Z/p\Z)^*.$$
We have $\ker(P_0)= U_1(\sqrt{d})$.
Hence, 
\begin{equation}\label{proof, equation of U(d) over U1(d)  in ramified case}
 |U(\sqrt{d})/U_1(\sqrt{d})|= p-1.   
\end{equation}
By \cite[Proposition 5.9]{Narkiewicz} we have $|U_1(\sqrt{d})/U_2(\sqrt{d})|=p$. Hence 
$$|U(\sqrt{d})/U_2(\sqrt{d})|=p(p-1).$$
Consider the commutative diagram in Figure \ref{commutative diagrams3}.
\begin{figure}
\centering
\begin{tikzcd}
&1 \arrow[d]     & 1 \arrow[d]    & 1 \arrow[d]&\\
1 \arrow[r]& U^{0}_2(\sqrt{d}) \arrow[d, "i"] \arrow[r, "i"]     & U_2(\sqrt{d}) \arrow[r,"N"] \arrow[d, "i"]    & U_1 \arrow[d, "i"]\arrow[ r]&1 \\
1 \arrow[r]& \ker(N) \arrow[d, "P"] \arrow[r, "i"]     & U(\sqrt{d}) \arrow[r,"N"] \arrow[d, "P"]    & U \arrow[d, "P"]& \\
1 \arrow[r]& \ker(N)/U^{0}_{2}(\sqrt{d}) \arrow[d] \arrow[r,"\overline{i}"]     & U(\sqrt{d})/U_{2}(\sqrt{d}) \arrow[r,"\overline{N}"] \arrow[d]    & U/U_{1} \arrow[d]& \\
&1                          & 1                       & 1 &\\
\end{tikzcd}
\caption{Relations between $U_1, U^{0}_2(\sqrt{d})$, and $U_2(\sqrt{d})$, for ramified extensions.}\label{commutative diagrams3}
\end{figure}
By the Snake Lemma we have $\textup{Im}(\overline{i})= \ker(\overline{N})$. $\textup{Im}(\overline{N})$ is the set of squares in $U/U_{1}\cong (\Z/p\Z)^*$, so that $|\textup{Im}(\overline{N})|= (p-1)/2$. Therefore
$$|\ker(N)/U^{0}_2(\sqrt{d})|=2p.$$ 
Hence,
$$|U^{0}(\sqrt{d})/U^{0}_2(\sqrt{d})|=4p.$$ 
\item The norm homomorphism $N : U_{2n}(\sqrt{d})\to U_n$, $n\in \N$, is surjective. In fact, for any $z\in U_n$ we have $w= \exp_p(1/2\cdot \log_p(z)) \in U_n\subseteq U_{2n}(\sqrt{d})$ and $N(w)=z$. For any $n\geq k$, consider the commutative diagram in Figure \ref{commutative diagrams4}.
\begin{figure}
\centering
\begin{tikzcd}
&1 \arrow[d]     & 1 \arrow[d]    & 1 \arrow[d]&\\
1 \arrow[r]& U^{0}_{2n}(\sqrt{d}) \arrow[d, "i"] \arrow[r, "i"]     & U_{2n}(\sqrt{d}) \arrow[r,"N"] \arrow[d, "i"]    & U_n \arrow[d, "i"]\arrow[ r]&1 \\
1 \arrow[r]& U^{0}_{2k}(\sqrt{d}) \arrow[d, "P"] \arrow[r, "i"]     & U_{2k}(\sqrt{d}) \arrow[r,"N"] \arrow[d, "P"]    & U_k \arrow[d, "P"]\arrow[ r]&1 \\
1 \arrow[r]& U^{0}_{2k}(\sqrt{d})/U^{0}_{2n}(\sqrt{d}) \arrow[d] \arrow[r,"\overline{i}"]     & U_{2k}(\sqrt{d})/U_{2n}(\sqrt{d}) \arrow[r,"\overline{N}"] \arrow[d]    & U_k/U_{n} \arrow[d]\arrow[ r]&1 \\
&1                          & 1                       & 1 &\\
\end{tikzcd}
\caption{Relations between $U^{0}_{2k}(\sqrt{d})$ and $U^{0}_{2n}(\sqrt{d})$, with $n\geq k$, for ramified extensions.}\label{commutative diagrams4}
\end{figure}
By the Snake Lemma, $\textup{Im}(\overline{i})= \ker(\overline{N})$ and $\overline{N}$ is surjective. It follows from \cite[Proposition 5.9]{Narkiewicz} that 
$$|U_n/U_{n+1}|= p, \qquad n=1,2,3,\ldots,$$
and
$$|U_n(\sqrt{d})/U_{n+1}(\sqrt{d})|=p, \qquad n=1,2,3,\ldots.$$
Therefore,
$$|U_k/U_n|= p^{n-k}, \qquad n\geq k\geq 1,$$
and
$$|U_{2k}(\sqrt{d})/U_{2n}(\sqrt{d})|=p^{2(n-k)}, \qquad n\geq k\geq 1.$$
Hence, 
$$|U^{0}_{2k}(\sqrt{d})/U^{0}_{2n}(\sqrt{d})= p^{n-k}, \qquad n\geq k\geq 1.$$
\end{enumerate}
\end{proof}

\begin{lemma}\label{monothetic group lemma ramified case second}
Let $\Q_p(\sqrt{d})$ be a ramified quadratic extension of $\Q_p$, where $p\geq 3$ is a prime. Then:
\begin{enumerate}
    \item $$|U^{0}(\sqrt{d})/U^{0}_1(\sqrt{d})|= 4;$$
    \item $$|U^{0}_{2k+1}(\sqrt{d})/U^{0}_{2n+1}(\sqrt{d})|= p^{n-k}, \qquad n\geq k\geq 0.$$
\end{enumerate}
\end{lemma}
\begin{proof}
\begin{enumerate}
\item By \eqref{proof, equation of U(d) over U1(d)  in ramified case}, we have
$$|U(\sqrt{d})/U_1(\sqrt{d})|= p-1.$$
The norm homomorphism $N : U_{1}(\sqrt{d})\to U_{1}$, is surjective. In fact, for any $z\in U_{1}$ we have $w= \exp_p(1/2\cdot \log_p(z)) \in U_{1}\subseteq U_{1}(\sqrt{d})$ and $N(w)=z$.
Consider the commutative diagram in Figure \ref{commutative diagrams5}.
\begin{figure}
\centering
\begin{tikzcd}
&1 \arrow[d]     & 1 \arrow[d]    & 1 \arrow[d]&\\
1 \arrow[r]& U^{0}_1(\sqrt{d}) \arrow[d, "i"] \arrow[r, "i"]     & U_1(\sqrt{d}) \arrow[r,"N"] \arrow[d, "i"]    & U_1 \arrow[d, "i"]\arrow[ r]&1 \\
1 \arrow[r]& \ker(N) \arrow[d, "P"] \arrow[r, "i"]     & U(\sqrt{d}) \arrow[r,"N"] \arrow[d, "P"]    & U \arrow[d, "P"]& \\
1 \arrow[r]& \ker(N)/U^{0}_{1}(\sqrt{d}) \arrow[d] \arrow[r,"\overline{i}"]     & U(\sqrt{d})/U_{1}(\sqrt{d}) \arrow[r,"\overline{N}"] \arrow[d]    & U/U_{1} \arrow[d]& \\
&1                          & 1                       & 1 &\\
\end{tikzcd}
\caption{Relations between $U_1, U^{0}_1(\sqrt{d})$, and $U_1(\sqrt{d})$, for ramified extensions.}\label{commutative diagrams5}
\end{figure}
By the Snake Lemma we have $\textup{Im}(\overline{i})= \ker(\overline{N})$. $\textup{Im}(\overline{N})$ is the set of squares in $U/U_{1}\cong (\Z/p\Z)^*$, so that $|\textup{Im}(\overline{N})|= (p-1)/2$. Therefore
$$|\ker(N)/U^{0}_1(\sqrt{d})|=2.$$ 
Hence,
$$|U^{0}(\sqrt{d})/U^{0}_1(\sqrt{d})|=4.$$ 

\item Consider the norm homomorphism $N : U_{2n+1}(\sqrt{d})\to U_{n+1}$ for $n=0,1,2,\ldots$. We first notice that, in fact, its image is contained in $U_{n+1}$. Indeed, let $z=1 + \pi^{2n+1}b\in U_{2n+1}(\sqrt{d})$. Write $b=b_1 + b_2 \pi$ for some $b_1, b_2\in \Z_p$. We have
$$N(z)= N((1+ \pi^{2n+2} b_2 + \pi^{2n+1}b_1))= (1 + \pi^{2n+2}b_2)^2 - \pi^{4n+2}b_1^2.$$
Clearly, $N(z)\in U_{n+1}$.
The norm homomorphism $N : U_{2n+1}(\sqrt{d})\to U_{n+1}$, for each $n=0,1,2,\ldots$ is surjective. In fact, for any $z\in U_{n+1}$ we have $w= \exp_p(1/2\cdot \log_p(z)) \in U_{n+1}\subseteq U_{2n+1}(\sqrt{d})$ and $N(w)=z$. For any $n\geq k$, consider the commutative diagram in Figure \ref{commutative diagrams6}.
\begin{figure}
\centering
\begin{tikzcd}
&1 \arrow[d]     & 1 \arrow[d]    & 1 \arrow[d]&\\
1 \arrow[r]& U^{0}_{2n+1}(\sqrt{d}) \arrow[d, "i"] \arrow[r, "i"]     & U_{2n+1}(\sqrt{d}) \arrow[r,"N"] \arrow[d, "i"]    & U_{n+1} \arrow[d, "i"]\arrow[ r]&1 \\
1 \arrow[r]& U^{0}_{2k+1}(\sqrt{d}) \arrow[d, "P"] \arrow[r, "i"]     & U_{2k+1}(\sqrt{d}) \arrow[r,"N"] \arrow[d, "P"]    & U_{k+1} \arrow[d, "P"]\arrow[ r]&1 \\
1 \arrow[r]& U^{0}_{2k+1}(\sqrt{d})/U^{0}_{2n+1}(\sqrt{d}) \arrow[d] \arrow[r,"\overline{i}"]     & U_{2k+1}(\sqrt{d})/U_{2n+1}(\sqrt{d}) \arrow[r,"\overline{N}"] \arrow[d]    & U_{k+1}/U_{n+1} \arrow[d]\arrow[ r]&1 \\
&1                          & 1                       & 1 &
\end{tikzcd}
\caption{Relations between $U^{0}_{2k+1}(\sqrt{d})$ and $U^{0}_{2n+1}(\sqrt{d})$, with $n\geq k$, for ramified extensions.}\label{commutative diagrams6}
\end{figure}
By the Snake Lemma, $\textup{Im}(\overline{i})= \ker(\overline{N})$ and $\overline{N}$ is surjective. Since
$$|U_n/U_{n+1}|= p, \qquad n=1,2,3,\ldots,$$
and
$$|U_n(\sqrt{d})/U_{n+1}(\sqrt{d})|=p, \qquad n=1,2,3,\ldots,$$
we have
$$|U_{k+1}/U_{n+1}|= p^{n-k}, \qquad n\geq k\geq 0,$$
and
$$|U_{2k+1}(\sqrt{d})/U_{2n+1}(\sqrt{d})|=p^{2(n-k)}, \qquad n\geq k\geq 0.$$
Hence, 
$$|U^{0}_{2k+1}(\sqrt{d})/U^{0}_{2n+1}(\sqrt{d})|= p^{n-k}, \qquad n\geq k\geq 0.$$
\end{enumerate}
\end{proof}

\subsection{Möbius Transformations Over $p$-adic Fields}
Let $K$ be either $\Q_p$ or a quadratic extension thereof. Denote by $P^1(K)$ the projective line over $K$, defined as a set 
$$P^1(K)= (K^2- \{(0,0)\})/ \sim,$$
where $\sim$ is the equivalence relation defined by: $(x,y)\sim (x',y')$ if and only if $x'=cx$, $y'=cy$ for some $c\in K^*$. Denote the equivalence class of $(x,y)$ by $[x,y]$.
The point $[1,0]$ is the {\it point at infinity}, denoted by $\infty$. The field $K$
is embedded into $P^1(K)$ by the map $x\to [x,1]$. The subset 
$$\{[x,1]\in P^1(K)\, | \,x\in K\}\subseteq P^1(K)$$
consists of all points of $P^1(K)$ except for $\infty$. We usually view $P^1(K)$ as $K\cup \{\infty\}$.
Define a metric on $K\cup \{\infty\}$ by 
$$\rho(z_1,z_2)= \frac{|z_1- z_2|_p}{\max\{|z_1|_p,1\}\cdot \max\{|z_2|_p,1\}}, \qquad z_1,z_2\in K,$$
and 
\begin{equation*}
    \begin{split}
        \rho(z,\infty)&= \begin{cases}
            1, &\qquad |z|_p\leq 1\\
            \frac{1}{|z|_p},&\qquad |z|_p> 1.
        \end{cases}
    \end{split}
\end{equation*}
A Möbius transformation of $P^1(K)$ is a transformation of the form 
\begin{equation}\label{prel sec, equation of Mobius transformation}
\phi(x)= \frac{ax+b}{cx+ d}, \qquad (a,b,c,d \in K,\, ad-bc\neq 0),    
\end{equation}
where we agree that $\phi(-d/c)=\infty$ and $\phi(\infty)= a/c$.
Every $p$-adic Möbius transformation $\phi$ may be written uniquely as a composition $\phi_1 \circ \phi_2 \circ \phi_3\circ \phi_4$, where $\phi_1(x)= x+ \beta$, $\phi_2(x)= \alpha x$, $\phi_3(x)= 1/x$, and $\phi_4(x) = x + \gamma$ for some $\alpha, \beta, \gamma \in K$ (see \cite[p.322]{Robert}).

Let $D= a+ \pi^k O_K$ be a disk in $K$ for some $a\in K$ and $k\in \Z$. The image of $D$ under $\phi_1$ is
\begin{equation*}
    \begin{split}
        \phi_1(D)&=\{a+ \pi^k y + \beta\, |\, y\in O_K\}
        = a+ \beta + \pi^k O_K,
    \end{split}
\end{equation*}
and under $\phi_2$
\begin{equation*}
    \begin{split}
        \phi_2(D)&=\{\alpha(a+\pi^k y), |\, y\in O_K\}
        = \{\alpha a+\pi^{\nu_{\pi}(\alpha)+k} y, |\, y\in O_K\}
        = \alpha a + \pi^{\nu_{\pi}(\alpha)+k} O_K.
    \end{split}
\end{equation*}
As to the image of $D$ under $\phi_3$, we distinguish between two cases:

$(i)$ $0 \in D$.

In this case $D=\pi^k O_K$, and thus:
$$\phi_3(D)= \left\{\frac{1}{x}\,:\, x\in \pi^kO_K\right\}=\pi^{-k+1} O_K^C,$$
where $O_K^C$ is the complement of $O_K$ in $K\cup \{\infty\}$.

$(ii)$ $0\notin D$.

Here:
\begin{equation*}
    \begin{split}
     \phi_3(D)&=\left\{\frac{1}{a+ \pi^k y}\, :\, y\in O_K\right\}
        = \left\{\frac{1}{a(1+ \pi^{\nu_{\pi}(1/a)+ k} y')} :\, y'\in O_K\right\}\\
         &= \left\{1/a \cdot(1+ \pi^{\nu_{\pi}(1/a)+ k} z)\, |\, z\in O_K\right\}
       = \left\{1/a+ \pi^{2\nu_{\pi}(1/a)+ k} z'\, |\, z'\in O_K\right\}\\
       &= 1/a+ \pi^{2\nu_{\pi}(1/a)+ k}O_K.
    \end{split}
\end{equation*}

\section{Examples}\label{examples section}
Our first example presents the closed subgroups of $U$ and $U(\sqrt{d})$, generated by various elements.
\begin{example}\emph{
\begin{enumerate}
\item Let $\lambda\in \Z_3$ be the square root of $244= 1+ 3^5$, that is $2$ modulo 3:
$$\lambda= -\exp_3(1/2\cdot \log_3(244)).$$
The order of $\lambda$ modulo $3$ is $l=2$, and $\lambda^2= 1+ 3^5$, so that $k=5$. By Theorem \ref{prop, orbit closure of 1 in Qp}:
$$\overline{\{\lambda^n\,|\, n\in \N\}}= \bigcup_{i=0}^1 \lambda^i \cdot (1+ 3^5\Z_3)= \pm1+ 3^5\Z_3.$$
\item Let $\lambda= 2+ \sqrt{5}\in \Q_3(\sqrt{5})$. The order of $\lambda$ modulo 3 is $l=8$, and $\lambda^8= 1+ 3^2(5760+ 25184\cdot \sqrt{5})$, so that $k=2$. By Theorem \ref{prop, orbit of 1 in ramified extension, under T(z)=lambda z}, 
$$\overline{\{\lambda^n\,|\, n\in \N\}}= \bigcup_{i=0}^7 \lambda^i\cdot U_{2}^{0}(\sqrt{5}),$$
which is a subgroup of index 3 of $U^0(\sqrt{5})$.
\item Let $\lambda= 2+ \sqrt{5}\in \Q_5(\sqrt{5})$. The order of $\lambda$ modulo $\sqrt{5}$ is $l=4$, and $\lambda^{4}=1 + \sqrt{5}(72 + 32\sqrt{5})$, so that $k=1$. By Theorem \ref{prop, orbit of 1 in extension2, under T(z)=lambda z}:
$$\overline{\{\lambda^n\,|\, n\in \N\}}= U^0(\sqrt{5}).$$
\end{enumerate}
}
\end{example}
\newpage
\begin{example}\label{repeted root example}\emph{By Theorem \ref{theorem, when single root}:
\begin{enumerate}
\item For $a_n= (1+ 3 n)\cdot 4^n$ over $\Q_3$: 
\begin{equation*}
\begin{split}
K&= 7 + 3^{2}\Z_3. 
\end{split}   
\end{equation*}
\item For $a_n= (1+n)\cdot4^n$ over $\Q_3$:
\begin{equation*}
\begin{split}
K&= 4(1+ 3\Z_3)^C  = (4 +3\Z_3)^C = (1 +3\Z_3)^C= 1 +3\Z_3^C.
\end{split}   
\end{equation*}
\end{enumerate}
}
\end{example}
The following three examples deal with the Kepler sets of some recurrence sequences, when the roots $\lambda_1$ and $\lambda_2$ belong to $\Q_p$ (Example \ref{example when root in Qp}), to an unramified extension (Example \ref{example of Kepler set in unramifed}), or to a ramified extension (Example \ref{example of Kepler set in ramifed}).
\begin{example}\label{example when root in Qp}\emph{
\begin{enumerate}
\item For $a_n= 1+ 4^n$ over $\Q_3$, by Theorem \ref{theorem for p-adic Kepler set of order 2 when roots are in Q_p}.$(1)$:
\begin{equation*}
\begin{split}
K&= 7 + 3^{2}\Z_3.
\end{split}   
\end{equation*}
\item For $a_n= -1+ 4^n$ over $\Q_3$, by Theorem \ref{theorem for p-adic Kepler set of order 2 when roots are in Q_p}.$(2)$:
\begin{equation*}
\begin{split}
K= (1+ 3\Z_3)^C=1+ 3\Z_3^C.
\end{split}   
\end{equation*}
\item For $a_n= -1+ 8^n$ over $\Q_3$, by Theorem \ref{theorem for p-adic Kepler set of order 2 when roots are in Q_p}.$(3)$:
\begin{equation*}
\begin{split}
K= 3^2\Z_3\cup \frac{1}{3}\cdot\Z_3^C.
\end{split}   
\end{equation*}
\end{enumerate}}
\end{example}

\begin{example}\label{example of Kepler set in unramifed}\emph{
\begin{enumerate}
\item For $a_n= (1-3\sqrt{2})(9+ \sqrt{2})^n + (1+3\sqrt{2})(9- \sqrt{2})^n$ over $\Q_3$, by Theorem \ref{theorem for p-adic Kepler set of order 2 when roots are in Q_p}.$(1)$:
\begin{equation*}
\begin{split}
K= (2/3+ \Z_3) \cup (3+ 3^2\Z_3).
\end{split}   
\end{equation*}
\item For $a_n= \frac{1}{3\sqrt{5}}\left(\frac{1+ 3\sqrt{5}}{2}\right)^n-  \frac{1}{3\sqrt{5}}\left(\frac{1- 3\sqrt{5}}{2}\right)^n$ over $\Q_3$, by Theorem \ref{theorem for p-adic Kepler set of order 2 when roots are in Q_p}.$(2)$:
\begin{equation*}
\begin{split}
K= (2+ 3\Z_p)^C.
\end{split}   
\end{equation*}
\item For $a_n= \frac{1}{\sqrt{2}}\left(9+\sqrt{2}\right)^n-  \frac{1}{\sqrt{2}}\left(9-\sqrt{2}\right)^n$ over $\Q_3$, by Theorem \ref{theorem for p-adic Kepler set of order 2 when roots are in Q_p}.$(3)$:
\begin{equation*}
\begin{split}
K= 3^2\Z_3\cup \frac{1}{3}\cdot \Z_3^C.
\end{split}   
\end{equation*}
\end{enumerate}}
\end{example}

\begin{example}\label{example of Kepler set in ramifed}\emph{
\begin{enumerate}
\item For $a_n= (1+ \sqrt{3})^n + (1-\sqrt{3})^n$ over $\Q_3$, by Theorem \ref{theorem for p-adic Kepler set2 of order 2 when roots are in Q_p}.$(1)$:
\begin{equation*}
\begin{split}
K= 1+ 3\Z_3.
\end{split}   
\end{equation*}
\item For $a_n= \left(\frac{1+ 5^2\sqrt{5}}{2}\right)^n- \left(\frac{1- 5^2\sqrt{5}}{2}\right)^n$ over $\Q_5$, by Theorem \ref{theorem for p-adic Kepler set2 of order 2 when roots are in Q_p}.$(2)$:
\begin{equation*}
\begin{split}
K= (3+ 5\Z_5)^C.
\end{split}   
\end{equation*}
\item For $a_n= (3+ \sqrt{3})^n + (3- \sqrt{3})^n$ over $\Q_3$, by Theorem \ref{theorem for p-adic Kepler set2 of order 2 when roots are in Q_p}.$(3)$:
\begin{equation*}
\begin{split}
K= \Q_3.
\end{split}   
\end{equation*}
\end{enumerate}}
\end{example}

\section{Proofs}
\subsection{Proof of Theorem \ref{prop, orbit closure of 1 in Qp}}
We have 
\begin{equation}\label{equa, of lambda to power mn}
\begin{split}
\lambda^{ln}&= \exp_p\left(\log_p(\lambda^{ln})\right)=\exp_p\left( n\log_p(\lambda^{l})\right)= \exp_p\left( n\log_p(1+ \mu p^k)\right),\\
\end{split}    
\end{equation}
where $\mu$ is such that $\lambda^{l}=1+ p^k\mu$ and $p\nmid \mu$.
Put $b= p^{-k}\log_p(1+\mu p^k)$, and note that $|b|_p=1$.   
By \eqref{equa, of lambda to power mn}, we get
 \begin{equation}\label{equa, of lambda to power mn closer}
    \begin{split}
        \overline{\{\lambda^{ln}\,|\, n\in\N\}}&= \left\{\exp_p\left( x p^k b\right)\,|\, x\in \Z_p\right\}=\exp_p(p^k\Z_p)= 1+ p^k\Z_p.\\
    \end{split}    
    \end{equation}
Clearly, the sets $\lambda^i+ p^k \Z_p$, $0\leq i\leq l-1$, are pairwise disjoint.   
Hence,
$$\overline{\{\lambda^{n}\,|\, n\in \N}\}= \bigsqcup_{i=0}^{l-1}(\lambda^i + p^k\Z_p).$$
\vspace{1cm}
\subsection{Proof of Theorem \ref{prop, orbit of 1 in ramified extension, under T(z)=lambda z}}

Let $z= 1+ p^kw \in U^{0}_k(\sqrt{d})$ with $p\nmid w$.
Suppose that the order of $z$ modulo $p^n$ is $m$. Since
$(1+ p^k w)^m\equiv1\, (\textup{mod}\ p^n)$, we have
$$mp^k w +\binom{m}{2}(wp^{k})^2+ \cdots+ \binom{m}{m}(p^k w)^m\equiv0\, (\textup{mod}\ p^n),$$
so that $p^{n-k}| m$. By Lemma \ref{monothetic group lemma unramified case}.$(2)$, we have 
$$|U^{0}_k(\sqrt{d})/U^{0}_n(\sqrt{d})|= p^{n-k}, \qquad n\geq k\geq 1.$$
Thus $m= p^{n-k}$.
It follows from \cite[Lemma 2.4.1]{Ribes} that the closed subgroup generated by $z$ is $U^{0}_k(\sqrt{d})$. In particular, for $z=\lambda^l$ we have
$$\overline{\{\lambda^{ln}\,|\,n\in \N\}}= U^{0}_k(\sqrt{d}).$$
Denote $W_i= (\lambda^i+ p^k\Z_p[\sqrt{d}])\bigcap U^{0}(\sqrt{d})$, $0\leq i\leq l-1$.
We have
\begin{equation}
\begin{split}
W_i&= \{y= \lambda^i + p^k (a+ b\sqrt{d})\, |\, a,b \in \Z_p,\, N(y)= \pm 1\}\\
&= \{y= \lambda^i(1 + p^k (a_1+ b_1\sqrt{d}))\, |\, a_1,b_1 \in \Z_p,\, N(y)= \pm 1\}\\
&= \lambda^i\cdot \{y_1= 1 + p^k (a_1+ b_1\sqrt{d})\, |\, a_1,b_1 \in \Z_p,\, N(y_1)= 1\}\\
&= \lambda^i\cdot U_k^0(\sqrt{d}),
\end{split}
\end{equation}
where the third equality is due to the fact that $\lambda \in U^0(\sqrt{d})$. Clearly, the sets $W_i$, $0\leq i\leq l-1$, are pairwise disjoint.
Hence, 
\begin{equation}\label{proof sec, equation of Wi in orbit theorem}
\overline{\{\lambda^{n}\,|\, n\in \N\}}= \bigsqcup_{i=0}^{l-1}W_i= \bigsqcup_{i=0}^{l-1}(\lambda^i+ p^k\Z_p[\sqrt{d}])\bigcap U^{0}(\sqrt{d}).    
\end{equation}
By Lemma \ref{monothetic group lemma unramified case}.$(1)$ and \eqref{proof sec, equation of Wi in orbit theorem},
if $k=1$ and $l=2(p+1)$, then
$$\overline{\{\lambda^{n}\,|\, n\in \N\}}=U^{0}(\sqrt{d}).$$
\vspace{0.25cm}

\subsection{Proof of Theorem \ref{prop, orbit of 1 in extension2, under T(z)=lambda z}}
We have $z = 1 + \pi^{k} \mu$ with $k>0, \pi\nmid \mu$.
Distinguish between two cases:
\vspace{0.25cm}

\noindent\underline{Case 1.}  $k$ is even, say, $k= 2s$ for some integer $s\geq 1$.
\vspace{0.25cm}

Let $z= 1+ \pi^{2 s}w \in U^{0}_{2s}(\sqrt{d})$ with $\pi\nmid w$.
Suppose that the order of $z$ modulo $\pi^{2n}$ is~$m$. Since $(1+ \pi^{2s} w)^m\equiv1(\textup{mod}\ \pi^{2n})$, we have
$$m \pi^{2s}w +\binom{m}{2}(\pi^{2s}w)^2+ \cdots+ \binom{m}{m}(\pi^{2s}w)^m\equiv0\ (\textup{mod}\ \pi^{2n}),$$
so that $p^{n-s}| m$. By Lemma \ref{monothetic group lemma ramified case}, we have
$$\left|U^{0}_{2s}(\sqrt{d})/U^{0}_{2n}(\sqrt{d})\right|= p^{n-s}, \qquad n\geq s\geq 1.$$
Thus, $m= p^{n-s}$. It follows from \cite[Lemma 2.4.1]{Ribes} that the closed subgroup generated by $z$ is $U^{0}_{2s}(\sqrt{d})$. In particular, for $z=\lambda^l$ we have
$$\overline{\{\lambda^{ln}\,|\,n\in \N\}}= U^{0}_{2s}(\sqrt{d}).$$
Denote $W_i= (\lambda^i+ \pi^{2s}\Z_p[\sqrt{d}])\bigcap U^{0}(\sqrt{d})$.
Clearly, $W_i= \lambda^iU^{0}_{2s}(\sqrt{d})$, and the sets $W_i, 0\leq i\leq l-1$ are pairwise disjoint.
Hence,
$$\overline{\{\lambda^{n}\,|\, n\in \N\}}= \bigsqcup_{i=0}^{l-1}W_i= \bigsqcup_{i=0}^{l-1}(\lambda^i+ \pi^{2s}\Z_p[\sqrt{d}])\bigcap U^{0}(\sqrt{d}).$$
\vspace{0.25cm}

\noindent\underline{Case 2.}  $k$ is odd, say, $k= 2s+1$ for some integer $s\geq 0$.
\vspace{0.25cm}

The proof in this case is very similar to that in the first one, except that we need to replace $U_{2s}(\sqrt{d})$, $U_{2n}(\sqrt{d})$, and Lemma \ref{monothetic group lemma ramified case} by $U_{2s+1}(\sqrt{d})$, $U_{2n +1}(\sqrt{d})$, and Lemma \ref{monothetic group lemma ramified case second}, respectively.

By Lemma \ref{monothetic group lemma ramified case second}.$(1)$ and \eqref{result, equation of monothetic group in ramified case},
if $k=1$ and $l=4$, then
$$\overline{\{\lambda^{n}\,|\, n\in \N\}}=U^{0}(\sqrt{d}).$$
\vspace{0.25cm}

\subsection{Proof of Theorem \ref{theorem, when single root}}
Let $c= c_1/c_2$. We have:
\begin{equation}\label{proof sec, ratio for single root}
\begin{split}
 \frac{a_{n+1}}{a_n}= \frac{(c_1+ (n+1) c_2)\lambda^{n+1}}{(c_1+ n c_2)\lambda^n}
    =\lambda \left(1+ \frac{1}{c_1/c_2+ n}\right)= \lambda\left(1+ \frac{1}{c+ n}\right).
\end{split}
\end{equation}
Define a Möbius transformation $T$ of  $\Q_p\cup \{\infty\}$ by 
\begin{equation}\label{proof sec, equation of T for repeated root}
   T(z)=  \lambda\left(1+ \frac{1}{c+ z}\right),\qquad z\in \Q_p\cup \{\infty\}.
\end{equation}
Thus, 
$$ K= T(\Z_p)= \lambda\left(1+ \frac{1}{c + \Z_p}\right).$$
\begin{enumerate}
    \item If $|c|_p>1$, then $1/c \in p\Z_p$, and therefore, 
\begin{equation}
    \begin{split}
       K &= \lambda\left(1+ 1/ c\cdot\frac{1}{1+ (1/c)\cdot \Z_p}\right)
        =\lambda\left(1+ 1/ c \cdot (1+ (1/c)\cdot \Z_p)\right)\\
        &= \lambda \left(1+ 1/ c +(1/c)^2\cdot \Z_p\right).\\
    \end{split}
\end{equation}
\item If $|c|_p\leq1$, then $c\in \Z_p$, and therefore,
\begin{equation}
    \begin{split}
        K&= \lambda\left(1+ \frac{1}{\Z_p}\right)
        = \lambda\left(1+ p\Z_p^C\right)
       = \lambda \left(1 + p\Z_p\right)^C.
    \end{split}
    \end{equation}
\end{enumerate}
\vspace{1cm}

\subsection{Proof of Theorem \ref{theorem for p-adic Kepler set of order 2 when roots are in Q_p} for $\lambda_1, \lambda_2 \in \Q_p$}
Let $c= c_1/c_2$. We have 
\begin{equation}\label{proof sec, ratio in the case when roots are in Q_p}
\begin{split}
    \frac{a_{n+1}}{a_n}= \frac{c_1\lambda_1^{n+1}+ c_2\lambda_2^{n+1}}{c_1\lambda_1^n + c_2\lambda_2^n}
    = \frac{c \lambda_1+ \lambda_2 (\lambda_2/\lambda_1)^n}{c+ (\lambda_2/\lambda_1)^n}.
\end{split}
\end{equation}
By Theorem \ref{prop, orbit closure of 1 in Qp},
\begin{equation}\label{proof sec, equation of orbit in the case of root in Q_p}
\overline{\{(\lambda_2/\lambda_1)^n\,|\, n\in\N\}} = \bigsqcup_{i=0}^{l-1} D_i,
\end{equation}
where 
$$D_i= (\lambda_2/\lambda_1)^i+ p^k \Z_p, \qquad 0\leq i\leq l-1.$$
Therefore, the Kepler set of $(a_n)_{n=0}^\infty$ is 
\begin{equation*}
\begin{split}
K= \left\{ \frac{c\lambda_1 + \lambda_2 z}{c+ z}\,: z\in \bigsqcup_{i=0}^{l-1} D_i\right\}.
\end{split}
\end{equation*}
In other words,
\begin{equation}\label{proof sec, Kepler set as the image of the union of Di}
K= T\left(\bigsqcup_{i=0}^{l-1} D_i\right)    
\end{equation}
where $T$ is the $p$-adic Möbius transformation given by 
\begin{equation}\label{proof sec, p-adic mob when root are in Q_p}
 T(z)=  \frac{c\lambda_1 + \lambda_2 z}{c+ z}= \lambda_2- \frac{c(\lambda_2-\lambda_1)}{c+ z}, \qquad z\in \Q_p\cup \{\infty\}.  
\end{equation}
Suppose first that $-c\notin D_i$, which happens for all $i$ in the first case, and for all $i\neq s$ in the range $[0,l-1]$ in the third case. Then $ \nu_p(c+ (\lambda_2/\lambda_1)^i)\leq k-1$. By \eqref{proof sec, p-adic mob when root are in Q_p} we have
\begin{equation}\label{proof sec, image of D_1 under T root in q_p}
\begin{split}
    T(D_i)& = T\left((\lambda_2/\lambda_1)^i + p^k \Z_p\right)
    = \lambda_2 - c(\lambda_2- \lambda_1)\cdot\frac{1}{c+ (\lambda_2/\lambda_1)^i + p^k \Z_p}\\
    &= \lambda_2 - \frac{c(\lambda_2-\lambda_1)}{c+ (\lambda_2/\lambda_1)^i}\cdot \frac{1}{1+ p^{-\nu_p\left(c+(\lambda_2/\lambda_1)^i\right)+k}\Z_p}\\
    &= \lambda_2 - \frac{c(\lambda_2-\lambda_1)}{c+ (\lambda_2/\lambda_1)^i}\cdot \left(1+ p^{-\nu_p\left(c+(\lambda_2/\lambda_1)^i\right)+k}\Z_p\right)\\
    &= \lambda_2 - \frac{c(\lambda_2-\lambda_1)}{c+ (\lambda_2/\lambda_1)^i} +p^{\nu_p(c(\lambda_2-\lambda_1))-2\nu_p\left(c+(\lambda_2/\lambda_1)^i\right)+k}\Z_p\\
    &= \frac{a_{i+1}}{a_{i}} + p^{\nu_p(c)+ \nu_p(\lambda_2-\lambda_1)-2\nu_p\left(c+(\lambda_2/\lambda_1)^i\right)+k}\Z_p.
\end{split}
\end{equation}
\begin{enumerate}
\item The proof follows immediately from \eqref{proof sec, Kepler set as the image of the union of Di} and \eqref{proof sec, image of D_1 under T root in q_p}.
\item Since $-c \in D_0$, we have $c+1\in p^k\Z_p$ and $\nu_p(c)=0$. Clearly, $\nu_p(\lambda_1/\lambda_2-1)=k$, with $k\geq 1$, and therefore $\lambda_1 \equiv \lambda_2\,(\textup{mod}\ p^{\nu_p(\lambda_2)+k})$. By \eqref{proof sec, Kepler set as the image of the union of Di} and \eqref{proof sec, p-adic mob when root are in Q_p} we have
\begin{equation}
\begin{split}
    K&= T(D_0)= \lambda_2- c(\lambda_2-\lambda_1)\cdot \frac{1}{c+ 1 + p^k \Z_p}= \lambda_2- c(\lambda_2-\lambda_1)\cdot \frac{1}{p^k \Z_p}\\
    &= \lambda_2+c\lambda_2\cdot \frac{1}{ \Z_p}=\lambda_2+ \lambda_2 p\Z_p^C= (\lambda_2+ \lambda_2p\Z_p)^C\\
    &= \left(\frac{\lambda_1+ \lambda_2}{2}+ p^{\nu_p(\lambda_2)+1}\Z_p\right)^C.\\
\end{split}
\end{equation}
\item In this case, $-c \in D_s$, but $-c\notin D_i$ for all other $i$ in the range $[0,l-1]$. The elements of $D_s$ are units, so that $\nu_p(c)=0$. By the definition of $l$ and $s$, for every $i\neq s$ in the range $[0, l-1]$ we have $(\lambda_2/\lambda_1)^i \not \equiv (\lambda_2/\lambda_1)^s \equiv -c\, (\textup{mod}\ p)$, and therefore $\nu_p(c+ (\lambda_2/\lambda_1)^i)= 0$.
By \eqref{proof sec, image of D_1 under T root in q_p} we have
\begin{equation}\label{proof sec, image of D_i under T root in q_p c_1/c_2 in Ds2}
\begin{split}
T(D_i)&= \frac{a_{i+1}}{a_{i}} +p^{\nu_p(\lambda_2-\lambda_1)+k}\Z_p,
\end{split} \qquad 0\leq i\leq l-1,\, i\neq s.
\end{equation}
Since $c+ (\lambda_2/\lambda_1)^s \in p^k\Z_p$ and $\nu_p(\lambda_1/\lambda_2-1)=0$, we obtain
\begin{equation}\label{proof sec, image of D_s under T root in q_p c_1/c_2 in Ds}
\begin{split}
  T(D_s)& = T\left((\lambda_2/\lambda_1)^s + p^k \Z_p\right)
    = \lambda_2 - c(\lambda_2- \lambda_1)\cdot\frac{1}{c+ (\lambda_2/\lambda_1)^s + p^k \Z_p}\\
    &= \lambda_2 -c(\lambda_2-\lambda_1)\cdot\frac{1}{p^k \Z_p}=\lambda_2+ \lambda_2\cdot\frac{1}{p^k \Z_p}= \lambda_2\left(1+ \frac{1}{p^{k-1}}\Z_p^C\right)\\
   &= p^{\nu_p(\lambda_2)+1-k}\Z_p^C.\\
    \end{split}
\end{equation}
The proof follows from \eqref{proof sec, Kepler set as the image of the union of Di}, \eqref{proof sec, image of D_i under T root in q_p c_1/c_2 in Ds2}, and \eqref{proof sec, image of D_s under T root in q_p c_1/c_2 in Ds}.
\end{enumerate}
\vspace{1cm}
\subsection{Proof of Theorem \ref{theorem for p-adic Kepler set of order 2 when roots are in Q_p} for $\lambda_1, \lambda_2 \notin \Q_p$}
Clearly, $c_1$ and $c_2$ are conjugate over $\Q_p$. Therefore, $\nu_p(c_1/c_2)=0$ and $N(c_1/c_2)=1$. Let $c= c_1/c_2$. We have 
\begin{equation}\label{proof sec, ratio in the case when roots are not Q_p}
\begin{split}
    \frac{a_{n+1}}{a_n}= \frac{c_1\lambda_1^{n+1}+ c_2\lambda_2^{n+1}}{c_1\lambda_1^n + c_2\lambda_2^n}
    = \frac{c\lambda_1+ \lambda_2 (\lambda_2/\lambda_1)^n}{c+ (\lambda_2/\lambda_1)^n}.
\end{split}
\end{equation}
Denote 
$$D_i= (\lambda_2/\lambda_1)^i + p^k \Z_p[\sqrt{d}], \qquad 0\leq i\leq l-1.$$
By the proof of Theorem \ref{prop, orbit of 1 in ramified extension, under T(z)=lambda z} we have
\begin{equation*}
    \overline{\left\{(\lambda_2/\lambda_1)^{ln +i}\,|\, n\in \N\right\}}= W_i,\qquad 0\leq i \leq l-1,
\end{equation*}
where 
\begin{equation}\label{proof sec, equation od Wi}
 W_i= \left((\lambda_2/\lambda_1)^i + p^k\Z_p[\sqrt{d}]\right)\cap U^{0}(\sqrt{d})= D_i \cap U^{0}(\sqrt{d}),   
\end{equation} 
and the sets $W_i$ are pairwise disjoint.
Therefore, the Kepler set of $(a_n)_{n=0}^\infty$ is
\begin{equation*}
K= \left\{ \frac{c\lambda_1 + \lambda_2 z}{c+ z}\,: z\in \bigsqcup_{i=0}^{l-1}W_i\right\}.    
\end{equation*}
In other words,
\begin{equation}\label{proof sec, image od union of Wi under T}
K= T\left(\bigsqcup_{i=0}^{l-1}W_i\right),
\end{equation}
where $T$ is the $p$-adic Möbius transformation given by
\begin{equation}\label{proof sec, p-adic mob when root are not in Q_p}
 T(z)=  \frac{c\lambda_1 + \lambda_2 z}{c+ z}= \lambda_2- \frac{c(\lambda_2-\lambda_1)}{c+ z}, \qquad z\in \Q_p(\sqrt{d})\cup \{\infty\}.  
\end{equation}
Suppose that $-c\notin W_i$, which happens for all $i$ in the first case, and for all $i\neq s$ in the range $[0,l-1]$ in the third case. Thus, $ \nu_p(c+ (\lambda_2/\lambda_1)^i)\leq k-1$. By \eqref{proof sec, p-adic mob when root are not in Q_p} we have
\begin{equation}\label{proof sec, image of D_i under T root in not in Q_p}
\begin{split}
    T(D_i)& = T\left((\lambda_2/\lambda_1)^i + p^k \Z_p[\sqrt{d}]\right)
    = \lambda_2 - \frac{c(\lambda_2- \lambda_1)}{c+ (\lambda_2/\lambda_1)^i + p^k \Z_p[\sqrt{d}]}\\
    &= \lambda_2- \frac{c(\lambda_2- \lambda_1)}{c+ (\lambda_2/\lambda_1)^i}\cdot \frac{1}{1+ p^{-\nu_p\left(c+ (\lambda_2/\lambda_1)^i\right)+k}  \Z_p[\sqrt{d}]}\\
    &= \lambda_2- \frac{c(\lambda_2- \lambda_1)}{c+ (\lambda_2/\lambda_1)^i}\cdot \left( 1+ p^{-\nu_p\left(c+ (\lambda_2/\lambda_1)^i\right)+k}  \Z_p[\sqrt{d}]\right)\\
    &= \lambda_2- \frac{c(\lambda_2- \lambda_1)}{c+ (\lambda_2/\lambda_1)^i} +  p^{\nu_p(c(\lambda_2-\lambda_1))-2\nu_p\left(c+ (\lambda_2/\lambda_1)^i\right)+k} \Z_p[\sqrt{d}]\\
    &= \frac{a_{i+1}}{a_i}+  p^{\nu_p(\lambda_2-\lambda_1)-2\nu_p\left(c+ (\lambda_2/\lambda_1)^i\right)+k} \Z_p[\sqrt{d}].
\end{split}
\end{equation}
By \eqref{proof sec, equation od Wi} and \eqref{proof sec, image of D_i under T root in not in Q_p}, we get
\begin{equation}\label{proof sec, unramified image of Wik}
\begin{split}
 T(W_i)&\subseteq T(D_i)\cap \Q_p=  \frac{a_{i+1}}{a_i}+ p^{\nu_p(\lambda_1-\lambda_2)-2\nu_p\left(c+\left(\lambda_2/\lambda_1\right)^i\right)+k}  \Z_p.
\end{split}
\end{equation}
We claim that the inclusion in \eqref{proof sec, unramified image of Wik} is actually an  equality. In fact, let $x$ belong to the right-hand side of \eqref{proof sec, unramified image of Wik}. By \eqref{proof sec, image of D_i under T root in not in Q_p}, there exists some $y\in D_i$ such that $T(y)= x$. Using~\eqref{proof sec, p-adic mob when root are not in Q_p}, we obtain
\begin{equation*}
\begin{split}
 y= -c(\lambda_1-x)/(\lambda_2-x).
\end{split}
\end{equation*}
Since $x\in \Q_p$, the elements $\lambda_1-x$ and $\lambda_2-x$ are conjugate over $\Q_p$, so that
\begin{equation*}
    \begin{split}
    N(y)&= N\left(\frac{-c(\lambda_1-x)}{\lambda_2-x}\right)
    = N\left(c\right)\cdot N\left(\frac{\lambda_1-x}{\lambda_2-x}\right)
    =1.
    \end{split}
\end{equation*}
Therefore, $y\in W_i$, so that 
\begin{equation}\label{proof, sec, image of Wi,k under T}
T(W_i)= \frac{a_{i+1}}{a_i}+p^{\nu_p(\lambda_1-\lambda_2)-2\nu_p\left(c+\left(\lambda_2/\lambda_1\right)^i\right)+k}  \Z_p.    
\end{equation}

\begin{enumerate}
\item The proof in this case follows immediately from \eqref{proof sec, image od union of Wi under T} and \eqref{proof, sec, image of Wi,k under T}.
\item By the choice of $l$, we have $k\geq 1$ and $\lambda_1 \equiv \lambda_2\,(\textup{mod}\ p^{\nu_p(\lambda_2)+k})$.
Since $-c \in W_0$, we have
\begin{equation}\label{proof sec, equation of Kepler set for m=1, second}
\begin{split}
 T(D_0)& = \lambda_2 - c(\lambda_2- \lambda_1)\cdot \frac{1}{c+ 1 + p^k \Z_p[\sqrt{d}]}\\
 &= \lambda_2 - c(\lambda_2 -\lambda_1)\cdot \frac{1}{p^k \Z_p[\sqrt{d}]}=  \lambda_2 + \lambda_2p^{\nu_p(\lambda_1/\lambda_2- 1) -k+ 1} \Z_p[\sqrt{d}]^C\\
 &= \lambda_2 + \lambda_2 \cdot p\Z_p[\sqrt{d}]^C=\frac{\lambda_1+\lambda_2}{2}+ p^{\nu_p(\lambda_2)+1}\Z_p[\sqrt{d}]^C\\
 &=\left(\frac{\lambda_1+\lambda_2}{2}+ p^{\nu_p(\lambda_2)+1}\Z_p[\sqrt{d}]\right)^C.\\
\end{split}
\end{equation}
By \eqref{proof sec, equation od Wi}, \eqref{proof sec, image od union of Wi under T}, and \eqref{proof sec, equation of Kepler set for m=1, second}, we obtain
\begin{equation}\label{proof sec, Kepler set as a subset for m=1}
\begin{split}
K &= T(W_0)\subseteq  T(D_0) \cap \Q_p\\
&=\left(\frac{\lambda_1+\lambda_2}{2}+ p^{ \nu_p(\lambda_2)+1}\Z_p\right)^C.\\
\end{split}
\end{equation}
As above, it can be easily seen that the inclusion in \eqref{proof sec, Kepler set as a subset for m=1} is actually an equality:
\begin{equation*}
\begin{split}
K &= \left(\frac{\lambda_1+\lambda_2}{2}+ p^{\nu_p(\lambda_2)+1}\Z_p\right)^C.
\end{split}
\end{equation*}
\item Recall that $-c\in W_s$, but $-c\notin W_i$ for all other $i$ in the range $[0,l-1]$. The elements of $W_s$ are units, so that $v_p(c)=0$. By the definition of $l$ and $s$, for every $i\neq s$ in the range $[0, l-1]$ we have $(\lambda_2/\lambda_1)^i \not \equiv (\lambda_2/\lambda_1)^s \equiv -c\,(\textup{mod}\ p)$, and therefore $v_p(c+ (\lambda_2/\lambda_1)^i)= 0$.
By \eqref{proof, sec, image of Wi,k under T}, we have
\begin{equation}\label{proof sec, image of Wi}
\begin{split}
T(W_i) &= \frac{a_{i+1}}{a_i}+ p^{\nu_p(\lambda_1-\lambda_2)+k}\Z_p,
\end{split} \qquad 0\leq i\leq l-1,\, i\neq s.
\end{equation}
Since $c+ (\lambda_2/\lambda_1)^s\in p^k\Z_p[\sqrt{d}]$ and $v_p(\lambda_1/\lambda_2-1)=0$, we get
\begin{equation}\label{proof sec, image of D_s under T root not in q_p c_1/c_2 in O}
\begin{split}
  T(D_s)& = T\left((\lambda_2/\lambda_1)^s + p^k \Z_p[\sqrt{d}]\right)\\
    &= \lambda_2 - c(\lambda_2- \lambda_1)\cdot\frac{1}{c+ (\lambda_2/\lambda_1)^s + p^k \Z_p[\sqrt{d}]}\\
    &= \lambda_2 -c(\lambda_2-\lambda_1)\cdot\frac{1}{p^k \Z_p[\sqrt{d}]}= \lambda_2+ \lambda_2 \frac{p^{v_p(\lambda_1/\lambda_2-1)}}{p^{k-1}}\cdot\Z_p[\sqrt{d}]^C\\
    &= \lambda_2 + \frac{\lambda_2}{p^{k-1}}\cdot \Z_p[\sqrt{d}]^C= \lambda_2\left(1 + \frac{1}{p^{k-1}}\Z_p[\sqrt{d}]^C\right)\\
    &= p^{\nu_p(\lambda_2)+1-k} \Z_p[\sqrt{d}]^C.
    \end{split}
\end{equation}

By \eqref{proof sec, equation od Wi} and \eqref{proof sec, image of D_s under T root not in q_p c_1/c_2 in O}, we obtain
\begin{equation}\label{proof sec, unramified case image of W_s under T} 
\begin{split}
    T(W_s)&\subseteq T(D_s) \cap \Q_p
    =p^{\nu_p(\lambda_2)+1-k} \Z_p^C.
\end{split}
\end{equation}
As above, the inclusion in \eqref{proof sec, unramified case image of W_s under T} is actually an equality:
\begin{equation}\label{proof sec, equality for Ws}
  \begin{split}
    T(W_s)
    &= p^{\nu_p(\lambda_2)+1-k} \Z_p^C.
\end{split}
\end{equation}
The proof follows from \eqref{proof sec, image od union of Wi under T}, \eqref{proof sec, image of Wi}, and \eqref{proof sec, equality for Ws}.
\end{enumerate}
\vspace{1cm}
\subsection{Proof of Theorem \ref{theorem for p-adic Kepler set2 of order 2 when roots are in Q_p}}
Recall that $c_1$ and $c_2$ are conjugate over $\Q_p$, and therefore, $\nu_p(c_1/c_2)=0$ and $N(c_1/c_2)=1$. Let $c= c_1/c_2$. We have 
\begin{equation*}\label{proof sec, ratio in the ramified}
\begin{split}
    \frac{a_{n+1}}{a_n}= \frac{c_1\lambda_1^{n+1}+ c_2\lambda_2^{n+1}}{c_1\lambda_1^n + c_2\lambda_2^n}
    = \frac{c\lambda_1+ \lambda_2 (\lambda_2/\lambda_1)^n}{c+ (\lambda_2/\lambda_1)^n}.
\end{split}
\end{equation*}
Denote 
$$D_i= (\lambda_2/\lambda_1)^i + \pi^k \Z_p[\sqrt{d}], \qquad 0\leq i\leq l-1.$$
By the proof of Theorem \ref{prop, orbit of 1 in extension2, under T(z)=lambda z}, we have
\begin{equation*}
    \overline{\left\{(\lambda_2/\lambda_1)^{ln +i}\,|\, n\in \N\right\}}= W_i,\qquad 0\leq i \leq l-1,
\end{equation*}
where 
\begin{equation}\label{proof sec, equation of Wi}
W_i= \left((\lambda_2/\lambda_1)^i + \pi^k\Z_p[\sqrt{d}]\right)\cap U^{0}(\sqrt{d})= D_i \cap U^{0}(\sqrt{d}),
\end{equation}
and the sets $W_i$ are pairwise disjoint.
Therefore, the Kepler set of $(a_n)_{n=0}^\infty$ is
\begin{equation*}
K= \left\{ \frac{c\lambda_1 + \lambda_2 z}{c+ z}\,: z\in \bigsqcup_{i=0}^{l-1}W_i\right\}.    
\end{equation*}
In other words,
\begin{equation}\label{proof sec, image od union of Wi under T in ramified case}
K= T\left(\bigsqcup_{i=0}^{l-1}W_i\right),
\end{equation}
where $T$ is the $p$-adic Möbius transformation given by
\begin{equation}\label{proof sec, p-adic mob in ramified case}
 T(z)=  \frac{c\lambda_1 + \lambda_2 z}{c+ z}= \lambda_2- \frac{c(\lambda_2-\lambda_1)}{c+ z}, \qquad z\in \Q_p(\sqrt{d})\cup \{\infty\}.  
\end{equation}
Suppose that $-c\notin W_i$, which happens for all $i$ in the first case, and for all $i\neq s$ in the range $[0,l-1]$ in the third case. Thus, $ \nu_{\pi}(c+ (\lambda_2/\lambda_1)^i)\leq k-1$. By \eqref{proof sec, p-adic mob in ramified case}, we get
\begin{equation}\label{proof sec, image of D_i in ramified case}
\begin{split}
    T(D_i)& = T\left((\lambda_2/\lambda_1)^i + \pi^k \Z_p[\sqrt{d}]\right)
    = \lambda_2 - \frac{c(\lambda_2- \lambda_1)}{c+ (\lambda_2/\lambda_1)^i + \pi^k \Z_p[\sqrt{d}]}\\
    &= \lambda_2- \frac{c(\lambda_2- \lambda_1)}{c+ (\lambda_2/\lambda_1)^i}\cdot \frac{1}{1+ p^{-\nu_p\left(c+ (\lambda_2/\lambda_1)^i\right)+(k/2)}  \Z_p[\sqrt{d}]}\\
    &= \lambda_2- \frac{c(\lambda_2- \lambda_1)}{c+ (\lambda_2/\lambda_1)^i}\cdot \left( 1+ p^{-\nu_p\left(c+ (\lambda_2/\lambda_1)^i\right)+(k/2)}  \Z_p[\sqrt{d}]\right)\\
    &= \lambda_2- \frac{c(\lambda_2- \lambda_1)}{c+ (\lambda_2/\lambda_1)^i} +  p^{\nu_p(c(\lambda_2-\lambda_1))-2\nu_p\left(c+ (\lambda_2/\lambda_1)^i\right)+(k/2)} \Z_p[\sqrt{d}]\\
    &= \frac{a_{i+1}}{a_i}+  p^{\nu_p(\lambda_2-\lambda_1)-2\nu_p\left(c+ (\lambda_2/\lambda_1)^i\right)+(k/2)} \Z_p[\sqrt{d}].
\end{split}
\end{equation}
By \eqref{proof sec, equation of Wi} and \eqref{proof sec, image of D_i in ramified case}, we obtain
\begin{equation}\label{proof sec, ramified image of Wik}
\begin{split}
 T(W_i)&\subseteq T(D_i)\cap \Q_p=  \frac{a_{i+1}}{a_i}+ p^{\lceil\nu_p(\lambda_1-\lambda_2)-2\nu_p\left(c+\left(\lambda_2/\lambda_1\right)^i\right)+(k/2)\rceil}  \Z_p.
\end{split}
\end{equation}
As in the proof of Theorem \ref{theorem for p-adic Kepler set of order 2 when roots are in Q_p} for $\lambda_1,\lambda_2 \not\in \Q_p$, the inclusion in \eqref{proof sec, ramified image of Wik} is actually an equality: 
\begin{equation}\label{proof, sec, image of Wi,k under T ramified}
T(W_i)= \frac{a_{i+1}}{a_i}+p^{\lceil\nu_p(\lambda_1-\lambda_2)-2\nu_p\left(c+\left(\lambda_2/\lambda_1\right)^i\right)+(k/2)\rceil}  \Z_p.    
\end{equation}
\begin{enumerate}
\item The proof in this case follows immediately from \eqref{proof sec, image od union of Wi under T in ramified case} and \eqref{proof, sec, image of Wi,k under T ramified}.
\item Using the facts that $k$ is an odd integer, that $k\geq 1$, and that $\lambda_1 \equiv \lambda_2\,(\textup{mod}\ p^{\nu_p(\lambda_2)+(k/2)})$, we have $\nu_p(\lambda_2) \in \Z$.
Since $-c \in W_0$, we have
\begin{equation}\label{proof sec, equation of Kepler set for m=1, ramified}
\begin{split}
 T(D_0)& = \lambda_2 - c(\lambda_2- \lambda_1)\cdot \frac{1}{c+ 1 + \pi^k \Z_p[\sqrt{d}]}\\
 &= \lambda_2 - c(\lambda_2 -\lambda_1)\cdot \frac{1}{\pi^k \Z_p[\sqrt{d}]}= \lambda_2 - c(\lambda_2 -\lambda_1)\cdot \frac{1}{\pi^{k-1}} \cdot \Z_p[\sqrt{d}]^C\\
 &=  \lambda_2 + \lambda_2(\lambda_1/\lambda_2- 1)\cdot \frac{1}{\pi^{k-1}}\Z_p[\sqrt{d}]^C\\
 &= \lambda_2 + \lambda_2 \cdot \pi\Z_p[\sqrt{d}]^C=\frac{\lambda_1+\lambda_2}{2}+ p^{\nu_p(\lambda_2)+(1/2)}\Z_p[\sqrt{d}]^C\\
 &=\left(\frac{\lambda_1+\lambda_2}{2}+ p^{\nu_p(\lambda_2)+(1/2)}\Z_p[\sqrt{d}]\right)^C.\\
\end{split}
\end{equation}
By \eqref{proof sec, equation of Wi}, \eqref{proof sec, image od union of Wi under T in ramified case}, and \eqref{proof sec, equation of Kepler set for m=1, ramified}, we get
\begin{equation}\label{proof sec, Kepler set as a subset for m=1 ramified}
\begin{split}
K &= T(W_0)\subseteq  T(D_0) \cap \Q_p\\
&=\left(\frac{\lambda_1+\lambda_2}{2}+ p^{\lceil \nu_p(\lambda_2)+1/2\rceil}\Z_p\right)^C.\\
\end{split}
\end{equation}
It can be easily seen that the inclusion in \eqref{proof sec, Kepler set as a subset for m=1 ramified} is actually an equality:
\begin{equation*}
\begin{split}
K &= \left(\frac{\lambda_1+\lambda_2}{2}+ p^{\lceil \nu_p(\lambda_2)\rceil+1}\Z_p\right)^C.
\end{split}
\end{equation*}
\item By the definition of $l$ and $s$, for every $i\neq s$ in the range $[0, l-1]$ we have $(\lambda_2/\lambda_1)^i \not \equiv (\lambda_2/\lambda_1)^s \equiv -c\,(\textup{mod}\ \pi)$, and therefore $v_p(c+ (\lambda_2/\lambda_1)^i)= 0$.
By \eqref{proof, sec, image of Wi,k under T ramified}, we have
\begin{equation}\label{proof sec, image of Wi ramified}
\begin{split}
T(W_i) &= \frac{a_{i+1}}{a_i}+ p^{\lceil\nu_p(\lambda_1-\lambda_2)+k/2\rceil}\Z_p,
\end{split} \qquad 0\leq i\leq l-1,\, i\neq s.
\end{equation}
Since $c+ (\lambda_2/\lambda_1)^s\in \pi^k\Z_p[\sqrt{d}]$ and $v_p(\lambda_1/\lambda_2-1)=0$, we get
\begin{equation}\label{proof sec, image of D_s ramified}
\begin{split}
  T(D_s)& = T\left((\lambda_2/\lambda_1)^s + \pi^k \Z_p[\sqrt{d}]\right)\\
    &= \lambda_2 - c(\lambda_2- \lambda_1)\cdot\frac{1}{c+ (\lambda_2/\lambda_1)^s + \pi^k \Z_p[\sqrt{d}]}\\
    &= \lambda_2 -c(\lambda_2-\lambda_1)\cdot\frac{1}{\pi^k \Z_p[\sqrt{d}]}= \lambda_2+ \lambda_2 \frac{p^{v_p(\lambda_1/\lambda_2-1)}}{\pi^{k-1}}\cdot\Z_p[\sqrt{d}]^C\\
    &= \lambda_2 + \frac{\lambda_2}{\pi^{k-1}}\cdot \Z_p[\sqrt{d}]^C= \lambda_2\left(1 + \frac{1}{\pi^{k-1}}\Z_p[\sqrt{d}]^C\right)\\
    &= p^{\nu_p(\lambda_2)-(k-1)/2} \cdot \Z_p[\sqrt{d}]^C.
    \end{split}
\end{equation}
By \eqref{proof sec, equation of Wi} and \eqref{proof sec, image of D_s ramified}, we obtain
\begin{equation}\label{proof sec, ramified case image of W_s under T} 
\begin{split}
    T(W_s)&\subseteq T(D_s) \cap \Q_p
    =p^{\lceil\nu_p(\lambda_2)+1/2-k/2\rceil} \Z_p^C.
\end{split}
\end{equation}
As above, the inclusion in \eqref{proof sec, ramified case image of W_s under T} is actually an equality:
\begin{equation}\label{proof sec, equality for Ws ramified}
  \begin{split}
    T(W_s)
    &= p^{\lceil\nu_p(\lambda_2)+1/2-k/2\rceil} \Z_p^C.
\end{split}
\end{equation}
The proof follows from \eqref{proof sec, image od union of Wi under T in ramified case}, \eqref{proof sec, image of Wi ramified}, and \eqref{proof sec, equality for Ws ramified}.
\end{enumerate}
\vspace{1cm}

\subsection{Proof of Theorem \ref{result, proposition on Kepler set equal to Qp}}
\noindent (1) Since the order of $\lambda_2/\lambda_1$ modulo $p$ is at most $p-1$, the result trivially follows from Theorem~\ref{theorem for p-adic Kepler set of order 2 when roots are in Q_p}.
\vspace{0.25 cm}

\noindent (2) By the definition of the rank of appearance, we have
\begin{equation*}
\begin{split}
L_{\alpha_{L}(p)}&\equiv 0\ (\textup{mod}\, p).
\end{split}
\end{equation*}
Since $L_{\alpha_{L}(p)}= \frac{\lambda_1^{\alpha_{L}(p)}- \lambda_2^{\alpha_{L}(p)}}{\lambda_1-\lambda_2}$, this implies
\begin{equation*}
\begin{split}
\lambda_1^{\alpha_{L}(p)}- \lambda_2^{\alpha_{L}(p)} &\equiv 0\ (\textup{mod}\, p),\\
\end{split}
\end{equation*}
so that 
\begin{equation*}
\begin{split}
\left(\frac{\lambda_2}{\lambda_1}\right)^{\alpha_{L}(p)} &\equiv 1\  (\textup{mod}\, p).\\
\end{split}
\end{equation*}
Hence, $\alpha_{L}(p)$ is the order of $\lambda_2/\lambda_1$ modulo $p$. It follows from Theorem \ref{theorem for p-adic Kepler set of order 2 when roots are in Q_p}.$(3)$ that the Kepler set of the sequence $(L_n)_{n=0}^\infty$ is
\begin{equation*}
\begin{split}
K((L_n))&= \bigsqcup_{i=1}^{\alpha_{F}(p)-1}\left(\frac{L_{i+1}}{L_i} + p^{\nu_p(\lambda_2-\lambda_1)+1} \Z_p\right)\bigsqcup \left(p^{\nu_p(\lambda_2)+1-1} \Z_p^C\right)\\
&= \bigsqcup_{i=1}^{p}\left(\frac{L_{i+1}}{L_i} + p\Z_p\right)\bigsqcup \Z_p^C.\\
\end{split}
\end{equation*} 
Since the sets $\frac{L_{i+1}}{L_i}+ p\Z_p$, $1\leq i\leq p$, are pairwise disjoint, we get $K((L_n))= \Q_p$.
\vspace{0.25 cm}

\noindent (3) By Remark \ref{result, lemma on k odd}, we have $l=2$. Note that $c_1/c_2=-1$. It follows from Theorem~\ref{theorem for p-adic Kepler set2 of order 2 when roots are in Q_p}.$(3)$ that the Kepler set is 
\begin{equation}\label{proof, equation for Kepler set equal to C in ramified}
\begin{split}
K((L_n))&= \left(\frac{L_{2}}{L_1} + p^{\lceil\nu_p(\lambda_2-\lambda_1)+ 1/2\rceil} \Z_p\right) \bigsqcup p^{\lceil\nu_p(\lambda_2)+1 -1/2\rceil}\Z_p^C\\
&= \left(\frac{L_{2}}{L_1} + p^{\lceil1/2+ 1/2\rceil} \Z_p\right) \bigsqcup p^{\lceil 1/2+1 -1/2\rceil}\Z_p^C\\
&= \left(\frac{L_{2}}{L_1} + p\Z_p\right) \bigsqcup \left( p \Z_p^C\right).\\
\end{split}   
\end{equation}
Now:
$$L_2= \frac{\lambda_2^2-\lambda_1^2}{\lambda_1- \lambda_2}= \lambda_1 + \lambda_2= 2 a.$$
Since $\nu_p(\lambda_1)=1/2$, we have $\nu_p(a)\geq 1$. Write $a=p a_1$. By \eqref{proof, equation for Kepler set equal to C in ramified},
\begin{equation}
\begin{split}
K((L_n))&= \left(2 p a_1 + p\Z_p\right) \bigsqcup \Z_p^* \bigsqcup \Z_p^C\\
&= p \Z_p \bigsqcup \Z_p^* \bigsqcup \Z_p^C\\
&=\Q_p.
\end{split}   
\end{equation}
\vspace{1cm}

\section{Proof of Theorem \ref{fibanacci theorem2}}
The roots of the characteristic polynomial of $L(a,1)$ are
$$\lambda_1= \frac{a+ \sqrt{D}}{2}, \qquad \lambda_2= \frac{a- \sqrt{D}}{2},$$
where $D=a^2 +4$. Since $\lambda_1\lambda_2=-1$ and $\lambda_1 + \lambda_2 =a\in \Z_p$, we have $|\lambda_1|_p= |\lambda_2|_p=1$. Let $m\in [1,\alpha_L(p)]$ be an arbitrary fixed integer. We have
\begin{equation}\label{proof, equation of consecutive ratio in lucas sequence}
\begin{split}
\frac{L_{n+m}}{L_{n}}= \frac{\lambda_1^{n+m}- \lambda_2^{n+m}}{\lambda_1^n -\lambda_2^n}= \frac{\lambda_1^m - \lambda_2^m (\lambda_2/\lambda_1)^n}{1- (\lambda_2/\lambda_1)^n}, \qquad n=0,1,2,\ldots.     
\end{split}   
\end{equation}
Denote
\begin{equation}\label{proof, equation of Km as closure}
\begin{split}
 K_m = \overline{\left\{ \frac{L_{n+m}}{L_n}\, :\, n\in \N\right\}}.    
\end{split}   
\end{equation}
We split the rest of the proof to three cases.

\vspace{0.25cm}
\noindent\underline{Case I.} $\lambda_1,\lambda_2 \in \Q_p$. Consider two subcases.
\vspace{0.25cm}

\noindent(1) $p\nmid D$.
\vspace{0.25cm}

By the definition of the rank of appearance, we have
\begin{equation*}
\begin{split}
 L_{\alpha_L(p)}&\equiv 0\ (\textup{mod}\, p),\\
 \frac{\lambda_1^{\alpha_L(p)}- \lambda_2^{\alpha_L(p)}}{\lambda_1-\lambda_2}&\equiv 0\ (\textup{mod}\, p),\\
  \frac{\lambda_1^{\alpha_L(p)}- \lambda_2^{\alpha_L(p)}}{\sqrt{D}}&\equiv 0\ (\textup{mod}\, p),\\
  \left(\frac{\lambda_2}{\lambda_1}\right)^{\alpha_L(p)}&\equiv 1 \ (\textup{mod}\, p). 
\end{split}    
\end{equation*}
Hence, $\alpha_L(p)$ is the order of $\lambda_2/\lambda_1$ modulo $p$. By Theorem \ref{prop, orbit closure of 1 in Qp}, we have
\begin{equation}\label{proof, equation of Di with out extension for Lucas}
\begin{split}
 \overline{\left\{\left(\lambda_2/\lambda_1\right)^n\, :\, n\in \N\right\}}= \bigsqcup_{i=0}^{\alpha_{L}(p)-1} D_i,
\end{split}    
\end{equation}
where $D_i= (\lambda_2/\lambda_1)^i + p \Z_p$.
By \eqref{proof, equation of consecutive ratio in lucas sequence}, \eqref{proof, equation of Km as closure}, and \eqref{proof, equation of Di with out extension for Lucas}, we get
\begin{equation*}
\begin{split}
K_m= \left\{ \frac{\lambda_1^{m}-\lambda_2^{m} z}{1-z}\,:\, z\in \bigsqcup_{i=0}^{\alpha_{L}(p)-1} D_i\right\}.
\end{split}
\end{equation*}
In other words,
\begin{equation}\label{proof sec, Lucas Kepler set as the image of the union of Di}
K_m= T_m\left(\bigsqcup_{i=0}^{\alpha_{L}(p)-1} D_i\right),
\end{equation}
where $T_m$ is the $p$-adic Möbius transformation given by 
\begin{equation}\label{proof sec, Lucas p-adic mob when root are in Q_p}
\begin{split}
  T_m(z)&=\frac{\lambda_1^{m}-\lambda_2^{m} z}{1-z}= \lambda_2^{m} + \frac{\lambda_1^{m}- \lambda_2^{m}}{1-z}=  \lambda_2^{m} + \frac{\sqrt{D}\cdot L_{m}}{1-z},
\end{split} \qquad z\in \Q_p\cup \{\infty\}.
\end{equation}
By \eqref{proof sec, Lucas p-adic mob when root are in Q_p}, we obtain
\begin{equation}\label{proof sec, Lucas Tm(D0) when lambda in Qp}
\begin{split}
 T_m(D_0)&=  \lambda_2^{m} + \frac{\sqrt{D}\cdot L_{m}}{p\Z_p}= \lambda_2^{m} +L_m\cdot \Z_p^C= (\lambda_2^m + L_m \cdot\Z_p)^C,\\
\end{split}
\end{equation}
and 
\begin{equation}\label{proof sec, Lucas Tm(Di) when lambda in Qp}
\begin{split}
T_m(D_i)&= \lambda_2^{m} +\frac{\lambda_1^{m}- \lambda_2^{m}}{1-(\lambda_2/\lambda_1)^i + p\Z_p}
= \lambda_2^{m} +\frac{\lambda_1^{m}- \lambda_2^{m}}{1-(\lambda_2/\lambda_1)^i }\cdot \frac{1}{1+ p\Z_p}\\
&= \lambda_2^{m} +\frac{\lambda_1^{m}- \lambda_2^{m}}{1-(\lambda_2/\lambda_1)^i }\cdot (1+ p\Z_p)
= \frac{L_{i+ m}}{L_i} +\frac{\lambda_1^{m}- \lambda_2^{m}}{1-(\lambda_2/\lambda_1)^i }\cdot p\Z_p\\
&= \frac{L_{i+ m}}{L_i} + L_m\cdot p\Z_p
\end{split} 
\end{equation}
for $1\leq i\leq \alpha_{L}(p)-1$.
In view of \eqref{proof sec, Lucas Tm(D0) when lambda in Qp} and \eqref{proof sec, Lucas Tm(Di) when lambda in Qp}, we rewrite \eqref{proof sec, Lucas Kepler set as the image of the union of Di} as
\begin{equation}\label{proof sec, equation2 of Km when lambda in Qp 2}
 K_m= \bigsqcup_{i=1}^{\alpha_{L}(p)-1}\left(\frac{L_{i+ m}}{L_i} + L_m\cdot p\Z_p\right)\bigsqcup(\lambda_2^m + L_m \cdot \Z_p)^C.  
\end{equation}
By \eqref{proof sec, equation2 of Km when lambda in Qp 2}, we conclude that 
\begin{equation}\label{proof sec, equation of Kapler when lambda in Qp Lucas}
\begin{split}
 K_{\alpha_{L}(p)}&\supseteq  (\lambda_2^{\alpha_{L}(p)} + L_{\alpha_{L}(p)}\Z_p)^C=(\lambda_2^{\alpha_{L}(p)} + p\Z_p)^C,\\
\end{split}
\end{equation}
and 
\begin{equation}\label{proof sec, equation2 of Kapler when lambda in Qp Lucas}
\begin{split}
 K_{\alpha_{L}(p)-2}&\supseteq \frac{L_{1+ \alpha_L(p)-2}}{L_1} + p\Z_p= \frac{\lambda_1^{\alpha_{L}(p)-1}- \lambda_2^{\alpha_L(p)-1}}{\lambda_1-\lambda_2} + p\Z_p= \lambda_2^{\alpha_{L}(p)} + p\Z_p.
\end{split}
\end{equation}
From \eqref{proof sec, equation of Kapler when lambda in Qp Lucas} and \eqref{proof sec, equation2 of Kapler when lambda in Qp Lucas} it follows that
$$K_{\alpha_{L}(p)}\bigcup K_{\alpha_{L}(p)-2}= \Q_p.$$
\vspace{0.25cm}

\noindent (2) $p^2\mid D$.
\vspace{0.25cm}

In this case the order of $\lambda_2/\lambda_1$ modulo $p$ is 1 and 
\begin{equation*}
  \lambda_2/\lambda_1= 1+ p^k \mu,
\end{equation*}
where $k=\nu_p(\sqrt{D})\geq 1$ and $p\nmid \mu$.
We have
\begin{equation}\label{proof sec, Lucas equation of D0 for l equal 1}
  \overline{\left\{\left(\lambda_2/\lambda_1\right)^n\, |\, n\in \N\right\}}= D_0'= 1 + p^k \Z_p.
\end{equation}
By \eqref{proof, equation of consecutive ratio in lucas sequence}, \eqref{proof, equation of Km as closure}, and \eqref{proof sec, Lucas equation of D0 for l equal 1}, we may write
\begin{equation}\label{proof sec, Lucas equation of Km in Qp for l equal 1}
\begin{split}
K_m= \left\{ \frac{\lambda_1^{m}-\lambda_2^{m} z}{1-z}\,:\, z\in  D_0'\right\}.
\end{split}
\end{equation}
In view of \eqref{proof sec, Lucas p-adic mob when root are in Q_p}, \eqref{proof sec, Lucas equation of D0 for l equal 1}, and \eqref{proof sec, Lucas equation of Km in Qp for l equal 1}, we obtain
\begin{equation}\label{proof sec, Lucas equation of Km for l equal to 1 in Qp}
\begin{split}
K_m= T_m (D_0')= \lambda_2^m + \frac{\sqrt{D}\cdot L_m}{p^k\Z_p}= \lambda_2^m + p L_m \Z_p^C= (\lambda_2^m + pL_m \Z_p)^C. \end{split}    
\end{equation}
By \eqref{proof sec, Lucas equation of Km for l equal to 1 in Qp}, we conclude that
\begin{equation*}
K_{\alpha_L(p)}= (\lambda_2^{\alpha_L(p)} + p L_{\alpha_L(p)}\Z_p)^C,\qquad   K_{\alpha_L(p)-2}= (\lambda_2^{\alpha_L(p)-2} + p L_{\alpha_L(p)-2}\Z_p)^C  
\end{equation*}
We claim that 
\begin{equation}\label{proof sec, Lucas intersection of two Kepler sets for l equal to 1}
  (\lambda_2^{\alpha_L(p)} + p L_{\alpha_L(p)}\Z_p) \bigcap (\lambda_2^{\alpha_L(p)-2} + p L_{\alpha_L(p)-2}\Z_p)= \varnothing. 
\end{equation}
In fact, suppose that the intersection in \eqref{proof sec, Lucas intersection of two Kepler sets for l equal to 1} is non-empty, which means that
\begin{equation*}
\begin{split}
\lambda_2^{\alpha_L(p)}& \equiv \lambda_2^{\alpha_L(p)-2} \ (\textup{mod}\, p),\\
\lambda_2^2&\equiv 1 \ (\textup{mod}\, p),\\
(a/2)^2&\equiv 1 \ (\textup{mod}\, p),\\
a^2&\equiv 4 \ (\textup{mod}\, p).
\end{split}
\end{equation*}
Now, since $p\mid D$, we have $a^2\equiv -4\  (\textup{mod}\, p)$, which gives a contradiction.
Hence,
$$K_{\alpha_L(p)}\bigcup K_{\alpha_L(p)-2}= \Q_p.$$

\vspace{0.25cm}
\noindent\underline{Case II.} $\lambda_1,\lambda_2 \in \Q_p(\sqrt{D})-\Q_p$, where $\Q_p(\sqrt{D})$ unramified. Consider two subcases. 
\vspace{0.25cm}

\noindent (1) $p\nmid D$.
\vspace{0.25cm}

As in Case I when $p\nmid D$, the order of $\lambda_2/\lambda_1$ modulo~$p$ is $\alpha_L(p)$. 
Denote
$$D_i= (\lambda_2/\lambda_1)^i + p \Z_p[\sqrt{D}], \qquad 0\leq i\leq \alpha_{L}(p)-1.$$
By Theorem \ref{prop, orbit of 1 in ramified extension, under T(z)=lambda z}, we have
\begin{equation}\label{proof sec, Lucas equation od Wi}
    \overline{\left\{(\lambda_2/\lambda_1)^{\alpha_{L}(p)n +i}\,|\, n\in \N\right\}}= W_i,\qquad 0\leq i \leq \alpha_{L}(p)-1,
\end{equation}
where 
\begin{equation}\label{proof sec, Lucas equation2 od Wi unramified}
W_i= \left((\lambda_2/\lambda_1)^i + p\Z_p[\sqrt{D}]\right)\bigcap U^{0}(\sqrt{D})= D_i \cap U^{0}(\sqrt{D}).    
\end{equation}
By \eqref{proof, equation of consecutive ratio in lucas sequence}, \eqref{proof, equation of Km as closure}, and \eqref{proof sec, Lucas equation od Wi}, we may write
\begin{equation*}
\begin{split}
K_m= \left\{ \frac{\lambda_1^{m}-\lambda_2^{m}\cdot z}{1-z}\,:\, z\in \bigsqcup_{i=0}^{\alpha_{L}(p)-1} W_i\right\}.
\end{split}
\end{equation*}
In other words,
\begin{equation}\label{proof sec, Lucas Kepler set as the image of the union of Wi}
K_m= T_m\left(\bigsqcup_{i=0}^{\alpha_{L}(p)-1} W_i\right),
\end{equation}
where $T_m$ is given by 
\begin{equation}\label{proof sec, Lucas p-adic mob when root are not  in Q_p}
\begin{split}
  T_m(z)&=\frac{\lambda_1^{m}-\lambda_2^{m}z}{1-z}= \lambda_2^{m} + \frac{\lambda_1^{m}- \lambda_2^{m}}{1-z}=  \lambda_2^{m} + \frac{\sqrt{D}\cdot L_{m}}{1-z},
\end{split} \qquad z\in \Q_p(\sqrt{D})\cup \{\infty\}.
\end{equation}
Clearly, both $\lambda_1^{\alpha_{L}(p)}$ and $\lambda_2^{\alpha_{L}(p)}$ belong to the subring $\Z_p + \sqrt{D}\cdot p\Z_p$ of $\Z_p[\sqrt{D}]$, say $\lambda_1^{\alpha_{L}(p)}= A + Bp\sqrt{D}$ and $\lambda_2^{\alpha_{L}(p)}= A - Bp\sqrt{D}$ for some $A, B \in \Z_p$. 
By \eqref{proof sec, Lucas p-adic mob when root are not  in Q_p} and the fact that $\nu_p(L_{\alpha_{L}(p)})=1$, we have
\begin{equation}\label{proof sec, Lucas Tm(D0) when lambda not in Qp}
\begin{split}
 T_m(D_0)&= \lambda_2^{m}+ \frac{\sqrt{D}\cdot L_{m}}{p\Z_p[\sqrt{D}]}=\lambda_2^{m} + L_{m}\cdot\Z_p[\sqrt{D}]^C\\
 &=\begin{cases}\Z_p[\sqrt{D}]^C, &\qquad 1\leq m<\alpha_{L}(p),\\
\left(A+ p \Z_p[\sqrt{D}]\right)^C, &\qquad m=\alpha_{L}(p).
 \end{cases}
\end{split}
\end{equation}
By \eqref{proof sec, Lucas equation2 od Wi unramified} and \eqref{proof sec, Lucas Tm(D0) when lambda not in Qp}, we conclude that
\begin{equation}\label{proof sec, Lucas image of W0 under Tm}
\begin{split}
 T_m(W_0)&\subseteq T_m(D_0)\cap \Q_p= \begin{cases}\Z_p^C, &\qquad 1\leq m<\alpha_{L}(p),\\
\left(A+ p \Z_p\right)^C, &\qquad m=\alpha_{L}(p).
 \end{cases}
\end{split}  
\end{equation}
It follows from the proof of Theorem \ref{theorem for p-adic Kepler set of order 2 when roots are in Q_p} that the inclusion in \eqref{proof sec, Lucas image of W0 under Tm} is an equality, so that
\begin{equation}\label{proof sec, Lucas image of W0 under Tm finnal}
\begin{split}
 T_m(W_0) &= \begin{cases}\Z_p^C, &\qquad 1\leq m<\alpha_{L}(p),\\
\left(  A+ p \Z_p\right)^C, &\qquad m= \alpha_{L}(p).
 \end{cases}  
 \end{split}
\end{equation}
According to \eqref{proof sec, Lucas p-adic mob when root are not  in Q_p}, similarly to \eqref{proof sec, Lucas Tm(Di) when lambda in Qp}, we obtain
\begin{equation}\label{proof sec, Lucas Tm(Di) when lambda not in Qp}
\begin{split}
T_m(D_i)= \frac{L_{i+ m}}{L_i} + L_m\cdot p\Z_p[\sqrt{D}],\qquad 0\leq i\leq \alpha_{L}(p)-1.
\end{split} 
\end{equation}
By \eqref{proof sec, Lucas equation2 od Wi unramified} and \eqref{proof sec, Lucas Tm(Di) when lambda not in Qp}, we have
\begin{equation}\label{proof sec, Lucas image of Wi under Tm}
\begin{split}
 T_m(W_i)&\subseteq T_m(D_i)\cap \Q_p=\frac{L_{i+ m}}{L_i}+L_m \cdot p \Z_p.
\end{split}
\end{equation}
Using the proof of Theorem \ref{theorem for p-adic Kepler set of order 2 when roots are in Q_p}, we get
\begin{equation}\label{proof sec, Lucas image of Wi under Tm finnal}
\begin{split}
 T_m(W_i) &= \frac{L_{i+ m}}{L_i}+L_m \cdot p \Z_p.
\end{split}  
\end{equation}
By \eqref{proof sec, Lucas Kepler set as the image of the union of Wi}, \eqref{proof sec, Lucas image of W0 under Tm finnal}, and \eqref{proof sec, Lucas image of Wi under Tm finnal}, we obtain
\begin{equation}\label{proof sec, Lucas equation of Km when lambda not in Qp 2}
\begin{split}
 K_m&=\bigcup_{i=1}^{\alpha_{L}(p)-1} \left(\frac{L_{i+ m}}{L_i} + L_m \cdot p\Z_p\right) \bigcup \begin{cases}\Z_p^C, &\qquad 1\leq m<\alpha_{L}(p),\\
\left( A+ p \Z_p\right)^C, &\qquad m=\alpha_{L}(p).
 \end{cases}    
\end{split} 
\end{equation}
It follows from \eqref{proof sec, Lucas equation of Km when lambda not in Qp 2} that 
\begin{equation*}\label{proof sec, Lucas equation of K alpha}
\begin{split}
 K_{\alpha_{L}(p)}&\supseteq \left(A+ p\Z_p\right)^C,\qquad K_{\alpha_{L}(p)-2}\supseteq \frac{L_{1+ \alpha_{L}(p)-2}}{L_1} + p\Z_p=A+ p\Z_p.\\
\end{split}
\end{equation*}
The last two inclusions yield
$$K_{\alpha_{L}(p)}\bigcup K_{\alpha_{L}(p)-2}= \Q_p.$$
\vspace{0.25cm}

\noindent (2) $p^2\mid D$.
\vspace{0.25cm}

In this case the order of $\lambda_2/\lambda_1$ modulo $p$ is 1 and 
\begin{equation}\label{proof sec,  Lucas equation of lambda1 by lambda2 in unramiffed for l equal 1}
\lambda_2/\lambda_1= 1+ p^k\mu,    
\end{equation}
where $k=\nu_p(\sqrt{D})\geq 1$ and $p\nmid \mu$. It follows from \cite[Lemma 2.1.(v)]{Sanna4} that $\alpha_L(p)=p$, and hence $\nu_p(L_{\alpha_L(p)})=1$ by Binet's formula.
Denote
$$D_0'= 1 + p^k \Z_p[\sqrt{d}],$$
where $d= D/p^{2k}$.
By Theorem \ref{prop, orbit of 1 in ramified extension, under T(z)=lambda z}, we have
\begin{equation}\label{proof sec, Lucas equation of D0 for l equal 1 in unramified}
  \overline{\left\{\left(\lambda_2/\lambda_1\right)^n\, |\, n\in \N\right\}}= W_0',
\end{equation}
where 
\begin{equation}\label{proof sec, Lucas equation2 od Wi unramified for l equal 1}
W_0'= D_0' \cap U^{0}(\sqrt{d}).    
\end{equation}
By \eqref{proof, equation of consecutive ratio in lucas sequence}, \eqref{proof, equation of Km as closure}, and \eqref{proof sec, Lucas equation of D0 for l equal 1 in unramified}, we have
\begin{equation}\label{proof sec, Lucas equation of Km in unramified for l equal 1}
 K_m = \left\{\frac{\lambda_1^m -\lambda_2^m z}{1-z}\, : \, z\in W_0'\right\}= T_m(W_0'),
\end{equation}
where $T_m$ is as in \eqref{proof sec, Lucas p-adic mob when root are not  in Q_p}.
Using the facts that $\lambda_2\equiv \lambda_1 (\textup{mod}\, p^k)$ with $k\geq 1$, $\alpha_L(p)=p$, and $\nu_p(L_{\alpha_L(p)})=1$, we have
\begin{equation}\label{proof sec, Lucas equation of Tm(D0) in unramified case}
 T_m(D_0')=  \lambda_2^m+ \frac{\sqrt{D}\cdot L_m}{p^k \Z_p[\sqrt{d}]}= \lambda_2^m + pL_m \Z_p[\sqrt{d}]^C= \left(\frac{\lambda_1^m + \lambda_2^m}{2} + p L_m \Z_p(\sqrt{d})\right)^C,  
\end{equation}
for $1\leq m\leq \alpha_L(p)$.
By \eqref{proof sec, Lucas equation2 od Wi unramified for l equal 1}, \eqref{proof sec, Lucas equation of Km in unramified for l equal 1}, and \eqref{proof sec, Lucas equation of Tm(D0) in unramified case}, we get
\begin{equation}\label{proof sec, Lucas image of W0 under Tm for l equal to 1}
\begin{split}
K_m=  T_m(W_0')&\subseteq T_m(D_0')\cap \Q_p= \left(\frac{\lambda_1^m + \lambda_2^m}{2} + p L_m \Z_p\right)^C.
\end{split}  
\end{equation}
It follows from the proof of Theorem \ref{theorem for p-adic Kepler set of order 2 when roots are in Q_p} that the inclusion in \eqref{proof sec, Lucas image of W0 under Tm for l equal to 1} is an equality:
\begin{equation}\label{proof sec, Lucas2 image of W0 under Tm for l equal to 1}
    K_m= \left(\frac{\lambda_1^m + \lambda_2^m}{2} + p L_m \Z_p\right)^C.
\end{equation}
By \eqref{proof sec, Lucas2 image of W0 under Tm for l equal to 1}, we conclude that
\begin{equation*}
K_{\alpha_L(p)}= \left(\frac{\lambda_1^{\alpha_L(p)} + \lambda_2^{\alpha_L(p)}}{2} + p L_{\alpha_L(p)} \Z_p\right)^C=\left(\frac{\lambda_1^{\alpha_L(p)} + \lambda_2^{\alpha_L(p)}}{2} + p^2 \Z_p\right)^C,
\end{equation*}
and 
\begin{equation*}
K_{\alpha_L(p)-2}= \left(\frac{\lambda_1^{\alpha_L(p)-2} + \lambda_2^{\alpha_L(p)-2}}{2} + p L_{\alpha_L(p)-2} \Z_p\right)^C=\left(\frac{\lambda_1^{\alpha_L(p)-2} + \lambda_2^{\alpha_L(p)-2}}{2} + p \Z_p\right)^C. 
\end{equation*}
We claim that
\begin{equation}\label{proof sec, Lucas equation for intersection of sets in unramified case}
\left(\frac{\lambda_1^{\alpha_L(p)} + \lambda_2^{\alpha_L(p)}}{2} + p^2 \Z_p\right)  \bigcap   \left(\frac{\lambda_1^{\alpha_L(p)-2} + \lambda_2^{\alpha_L(p)-2}}{2} + p \Z_p\right)= \varnothing.
\end{equation}
In fact, suppose that the intersection in \eqref{proof sec, Lucas equation for intersection of sets in unramified case} is non-empty, which means that
\begin{equation*}
\begin{split}
 \frac{\lambda_1^{\alpha_L(p)} + \lambda_2^{\alpha_L(p)}}{2}&\equiv \frac{\lambda_1^{\alpha_L(p)-2} + \lambda_2^{\alpha_L(p)-2}}{2}\ ({\textup{mod}\ p}),\\
\lambda_2^2&\equiv 1 \ (\textup{mod}\, p),\\
(a/2)^2&\equiv 1 \ (\textup{mod}\, p),\\
a^2&\equiv 4 \ (\textup{mod}\, p).
\end{split}
\end{equation*}
Now, since $p\mid D$, we have $a^2\equiv -4\  (\textup{mod}\, p)$, which gives a contradiction. Hence
$$K_{\alpha_L(p)}\bigcup K_{\alpha_L(p)-2}= \Q_p.$$

\vspace{0.25cm}
\noindent\underline{Case III.} $\lambda_1,\lambda_2 \in \Q_p(\sqrt{D})-\Q_p$, where $\Q_p(\sqrt{D})$ is ramified. Consider two subcases.
\vspace{0.25cm}

\noindent (1) $p\pdiv D$.
\vspace{0.25cm}

In this case $\nu_{\pi}(\sqrt{D})=1$. Therefore, the order of $\lambda_2/\lambda_1$ modulo~$\pi$ is 1. We have
\begin{equation}\label{proof sec,  Lucas equation of lambda1 by lambda2 in ramiffed for l equal 1 first}
\lambda_2/\lambda_1 = 1 + \pi \mu,\qquad (\pi\nmid \mu).    
\end{equation}
It follows from \cite[Lemma 2.1.(v)]{Sanna4} that $\alpha_L(p)=p$, and hence $\nu_p(L_{\alpha_L(p)})=1$ by Binet's formula.
Denote
$$D_0= 1 + \pi \Z_p[\sqrt{D}].$$
By the proof of Theorem \ref{theorem for p-adic Kepler set2 of order 2 when roots are in Q_p}, we know that
\begin{equation}\label{proof, equation of W0 in theorem of two rows}
 \overline{\left\{(\lambda_2/\lambda_1)^n \, :\, n\in \N\right\}}= W_0,   
\end{equation}
where
\begin{equation}\label{proof, Lucuas equation of W0 in theorem of two rows in ramified case}
W_0=D_0 \cap U^0(\sqrt{D}).    
\end{equation}
By \eqref{proof, equation of consecutive ratio in lucas sequence}, \eqref{proof, equation of Km as closure}, and \eqref{proof, equation of W0 in theorem of two rows}, we have
\begin{equation}\label{proof, equation of Km as image of W0 in ramified case}
\begin{split}
K_m &= \left\{\frac{\lambda_1^m -\lambda_2^m z}{1-z}\, :\, z\in W_0\right\}.
\end{split}
\end{equation}
In other words,
\begin{equation}\label{proof sec, Lucas equation1 of Km in ramified case}
  K_m= T_m(W_0),  
\end{equation}
where $T_m$ is the $p$-adic Möbius transformation given by
\begin{equation}\label{proof sec, Lucas equation1 p-adic mob in ramified case}
\begin{split}
  T_m(z)&=\frac{\lambda_1^{m}-\lambda_2^{m}z}{1-z}= \lambda_2^{m} + \frac{\lambda_1^{m}- \lambda_2^{m}}{1-z}=  \lambda_2^{m} + \frac{\sqrt{D}\cdot L_{m}}{1-z},
\end{split} \qquad z\in \Q_p(\sqrt{D})\cup \{\infty\}.
\end{equation}
We have 
\begin{equation}\label{proof equation of D0 under Tm in theorem of lucas}
 \begin{split}
  T_m(D_0)&= \lambda_2^m + \frac{\sqrt{D} \cdot L_m }{\pi \Z_p[\sqrt{D}]}  
  = \lambda_2^m + \pi L_m \Z_p[\sqrt{D}]^C\\
  &=\frac{\lambda_1^m + \lambda_2^m}{2}+ \pi L_m \Z_p[\sqrt{D}]^C
  =\left(\frac{\lambda_1^m + \lambda_2^m}{2}+ \pi L_m \Z_p[\sqrt{D}]\right)^C.\\
 \end{split}   
\end{equation}
By \eqref{proof, Lucuas equation of W0 in theorem of two rows in ramified case}, \eqref{proof, equation of Km as image of W0 in ramified case} and \eqref{proof equation of D0 under Tm in theorem of lucas}, we have
\begin{equation}\label{proof, equation of K_m in ramified for Lucas}
K_m \subseteq T_m(D_0)\cap \Q_p = \left(\frac{\lambda_1^m + \lambda_2^m}{2}+ pL_m \Z_p\right)^C.    
\end{equation}
It follows from the proof of Theorem \ref{theorem for p-adic Kepler set2 of order 2 when roots are in Q_p} that the inclusion in \eqref{proof, equation of K_m in ramified for Lucas} is an equality:
\begin{equation}\label{proof, equation2 of K_m in ramified for Lucas}
K_m = \left(\frac{\lambda_1^m + \lambda_2^m}{2}+ pL_m \Z_p\right)^C.   
\end{equation}
By \eqref{proof, equation2 of K_m in ramified for Lucas}, we have
\begin{equation}\label{proof, equation of Kalpha in ramified case Lucas}
\begin{split}
K_{\alpha_{L}(p)}&= \left(\frac{\lambda_1^{\alpha_{L}(p)} + \lambda_2^{\alpha_{L}(p)}}{2} + p L_{\alpha_{L}(p)} \Z_p\right)^C
=\left(\frac{\lambda_1^{\alpha_{L}(p)} + \lambda_2^{\alpha_{L}(p)}}{2} + p^2 \Z_p\right)^C,\\
\end{split}    
\end{equation}
and
\begin{equation}\label{proof, equation2 of Kalpha in ramified case Lucas}
\begin{split}
K_{\alpha_L(p)-2}= \left(\frac{\lambda_2^{\alpha_L(p)-2} + \lambda_1^{\alpha_L(p)-2}}{2} + p L_{\alpha_L(p)-2} \Z_p\right)^C=\left(\frac{\lambda_2^{\alpha_L(p)-2} + \lambda_1^{\alpha_L(p)-2}}{2} + p \Z_p\right)^C.    
\end{split}    
\end{equation}
Since $\alpha_L(p)>2$, we have $p\nmid L_2= a$. Therefore,
\begin{equation}\label{proof, equation on empty intersection of two balls}
 \left(\frac{\lambda_1^{\alpha_{L}(p)} + \lambda_2^{\alpha_{L}(p)}}{2} + p^2 \Z_p\right)\bigcap \left(\frac{\lambda_1^{\alpha_{L}(p)-2} + \lambda_2^{\alpha_{L}(p)-2}}{2} + p\Z_p\right)= \varnothing.   
\end{equation}
In fact, suppose that the intersection in \eqref{proof, equation on empty intersection of two balls} is non-empty, which means that
\begin{equation}\label{proof, equation to prove the claim}
 \frac{\lambda_1^{\alpha_{L}(p)} + \lambda_2^{\alpha_{L}(p)}}{2} \equiv  \frac{\lambda_1^{\alpha_{L}(p)-2} + \lambda_2^{\alpha_{L}(p)-2}}{2}\ (\textup{mod}\, p). 
 \end{equation}
 Using the facts that $\lambda_1\lambda_2=-1$ and that $\left(\lambda_2/ \lambda_1\right)^{\alpha_{L}(p)} \equiv 1\ (\textup{mod}\, p)$, as well as \eqref{proof, equation to prove the claim}, we get the following chain of congruences:
\begin{equation*}
\begin{split}
 \lambda_1^2 + \lambda_2^2& \equiv2 \ (\textup{mod}\,p),\\
 \lambda_1^2 + \lambda_2^2- 2\lambda_1\lambda_2& \equiv 4 \ (\textup{mod}\,p),\\
 (\lambda_1 - \lambda_2)^2& \equiv 4 \ (\textup{mod}\,p),\\
  D& \equiv 4 \ (\textup{mod}\,p),\\
   a^2& \equiv 0 \ (\textup{mod}\,p).\\
\end{split}   
\end{equation*}
This contradicts the fact that $p\nmid a$, and thus proves \eqref{proof, equation on empty intersection of two balls}. Hence,
$$K_{\alpha_{L}(p)}\bigcup K_{\alpha_{L}(p)-2}= \Q_p.$$
\vspace{0.25cm}

\noindent (2) $p^2\mid D$.
\vspace{0.25cm}

Since $a\in \Z_p^*$, the order of $\lambda_2/\lambda_1$ modulo $p$ is 1 and
\begin{equation}\label{proof sec,  Lucas equation of lambda1 by lambda2 in ramiffed for l equal 1 second}
\lambda_2/\lambda_1= 1+ \pi^k \mu,    
\end{equation}
where $k= \nu_{\pi}(\sqrt{D})\geq 3$ and $\pi\nmid \mu$. It follows from \cite[Lemma 2.1.(v)]{Sanna4} that $\alpha_L(p)=p$, and hence $\nu_p(L_{\alpha_L(p)})=1$ by Binet's formula.
Denote
$$D_0'= 1 + \pi^k \Z_p[\sqrt{d}],$$
where $d= D/p^{k}$.
By the proof of Theorem \ref{theorem for p-adic Kepler set2 of order 2 when roots are in Q_p}, we know that
\begin{equation}\label{proof, equation2 of W0 in theorem of two rows}
 \overline{\left\{(\lambda_2/\lambda_1)^n \, :\, n\in \N\right\}}= W_0',   
\end{equation}
where
\begin{equation}\label{proof, Lucuas equation2 of W0 in theorem of two rows in ramified case}
W_0'=D_0' \cap U^0(\sqrt{d}).    
\end{equation}
By \eqref{proof, equation of consecutive ratio in lucas sequence}, \eqref{proof, equation of Km as closure}, and \eqref{proof, equation2 of W0 in theorem of two rows}, we have
\begin{equation}\label{proof, equation2 of Km as image of W0 in ramified case}
\begin{split}
K_m &= \left\{\frac{\lambda_1^m -\lambda_2^m z}{1-z}\, :\, z\in W_0\right\}= T_m (W_0'),
\end{split}
\end{equation}
where $T_m$ as in \eqref{proof sec, Lucas equation1 p-adic mob in ramified case}.
We have 
\begin{equation}\label{proof equation2 of D0 under Tm in theorem of lucas}
 \begin{split}
  T_m(D_0')&= \lambda_2^m + \frac{\sqrt{D} \cdot L_m }{\pi^k \Z_p[\sqrt{d}]} = \lambda_2^m + \pi L_m \Z_p[\sqrt{d}]^C\\
  &= \frac{\lambda_1^m + \lambda_2^m}{2}+ \pi L_m \Z_p[\sqrt{d}]^C=\left(\frac{\lambda_1^m + \lambda_2^m}{2}+ \pi L_m \Z_p[\sqrt{d}]\right)^C.\\
 \end{split}   
\end{equation}
By \eqref{proof, Lucuas equation2 of W0 in theorem of two rows in ramified case}, \eqref{proof, equation2 of Km as image of W0 in ramified case}, and \eqref{proof equation2 of D0 under Tm in theorem of lucas}, we have
\begin{equation}\label{proof, equation22 of K_m in ramified for Lucas}
K_m \subseteq T_m(D_0')\cap \Q_p = \left(\frac{\lambda_1^m + \lambda_2^m}{2}+ pL_m \Z_p\right)^C.    
\end{equation}
Using the proof of Theorem \ref{theorem for p-adic Kepler set2 of order 2 when roots are in Q_p}, we get
\begin{equation}\label{proof, equation23 of K_m in ramified for Lucas}
K_m = \left(\frac{\lambda_1^m + \lambda_2^m}{2}+ pL_m \Z_p\right)^C.   
\end{equation}
By \eqref{proof, equation23 of K_m in ramified for Lucas}, we have
\begin{equation*}\label{proof, equation of Kalpha in ramified case Lucas2}
\begin{split}
K_{\alpha_{L}(p)}&= \left(\frac{\lambda_1^{\alpha_{L}(p)} + \lambda_2^{\alpha_{L}(p)}}{2} + p L_{\alpha_{L}(p)} \Z_p\right)^C=\left(\frac{\lambda_1^{\alpha_{L}(p)} + \lambda_2^{\alpha_{L}(p)}}{2} + p^2\Z_p\right)^C,
\end{split}    
\end{equation*}
and
\begin{equation*}\label{proof, equation2 of Kalpha in ramified case Lucas2}
\begin{split}
K_{\alpha_{L}(p)-2}= \left(\frac{\lambda_1^{\alpha_{L}(p)-2} + \lambda_2^{\alpha_{L}(p)-2}}{2} + p L_{\alpha_{L}(p)-2}\Z_p\right)^C=\left(\frac{\lambda_1^{\alpha_{L}(p)-2} + \lambda_2^{\alpha_{L}(p)-2}}{2} + p\Z_p\right)^C.    
\end{split}    
\end{equation*}
As in case II when $p^2\mid D$, we have
\begin{equation}\label{proof, equation2 on empty intersection of two balls}
 \left(\frac{\lambda_1^{\alpha_{L}(p)} + \lambda_2^{\alpha_{L}(p)}}{2} + p^2 \Z_p\right)\bigcap \left(\frac{\lambda_1^{\alpha_{L}(p)-2} + \lambda_2^{\alpha_{L}(p)-2}}{2} + p\Z_p\right)= \varnothing.   
\end{equation}
Hence,
$$K_{\alpha_{L}(p)}\bigcup K_{\alpha_{L}(p)-2}= \Q_p.$$
\vspace{1 cm}

\bibliographystyle{plain}

\end{document}